\documentclass[11pt]{article}
\usepackage[top=1.0in,bottom=1.0in,left=1.0in,right=1.0in]{geometry}
\usepackage{amsmath,amssymb,amsthm}
\usepackage{caption}
\usepackage{subcaption}
\usepackage{graphicx}
\usepackage[authoryear]{natbib}
\usepackage{enumitem}
\usepackage{accents}
\usepackage{setspace}
\usepackage{color}
\newcommand\T{\rule{0pt}{2.6ex}}       
\newcommand\B{\rule[-1.2ex]{0pt}{0pt}} 

\numberwithin{equation}{section}

\newtheorem{lemma}{Lemma}
\newtheorem{theorem}{Theorem}
\newtheorem{proposition}{Proposition}

\theoremstyle{remark}
\newtheorem*{remark}{Remark}

\begin{document}

\title{\textbf{Diffusion approximation for efficiency-driven queues:\\
A space-time scaling approach}}

\author{Shuangchi He\thanks{Department of Industrial and Systems Engineering,
National University of Singapore, heshuangchi@nus.edu.sg}}

\date{September 16, 2015}
\maketitle

\begin{abstract}
	
Motivated by call center practice, we propose a tractable model for
$\mbox{GI}/\mbox{GI}/n+\mbox{GI}$ queues in the efficiency-driven (ED) regime.
We use a one-dimensional diffusion process to approximate the virtual waiting
time process that is scaled in both space and time, with the number of servers
and the mean patience time as the respective scaling factors. Using this
diffusion model, we obtain the steady-state distributions of virtual waiting
time and queue length, which in turn yield simple formulas for performance
measures such as the service level and the effective abandonment fraction. These
formulas are generally accurate when the mean patience time is several times
longer than the mean service time and the patience time distribution does not
change rapidly around the mean virtual waiting time. For practical purposes,
these formulas outperform existing results that rely on the exponential service
time assumption.

To justify the diffusion model, we formulate an asymptotic framework by
considering a sequence of queues, in which both the number of servers and the
mean patience time go to infinity. We prove that the space-time scaled virtual
waiting time process converges in distribution to the one-dimensional diffusion
process. A fundamental result for proving the diffusion limit is a functional
central limit theorem (FCLT) for the superposition of renewal processes. We
prove that the superposition of many independent, identically distributed
stationary renewal processes, after being centered and scaled in space and time,
converges in distribution to a Brownian motion. As a useful technical tool, this
theorem characterizes the service completion process in heavy traffic, allowing
us to greatly simplify the many-server analysis when service times follow a
general distribution.

\end{abstract}

\section{Introduction}
\label{sec:Introduction}

Queues with many parallel servers are building blocks for modeling call center
operations; see \citet{GansETAL03} and \citet{AksinETAL07a} for comprehensive
reviews. A large call center faces a great amount of traffic that is stochastic
and time-varying. Since the rate of incoming calls changes over time, the system
may become overloaded during peak hours. Waiting on a phone line, a customer may
hang up before being connected to an agent. This phenomenon, referred to as
customer abandonment, is present in almost all call centers and becomes
prominent when the system is overloaded. A call center may also be intentionally
operated in an overloaded regime. Nowadays, more and more firms outsource their
call centers to save costs. In service-oriented call centers, staffing costs
usually dominate the expenses of customer delay and abandonment. As pointed out
by \citet{Whitt06}, the rational operational regime for these call centers is
the \emph{efficiency-driven (ED)} regime that emphasizes capacity utilization
over the quality of service. In the ED regime, the service capacity is set below
the customer arrival rate by a moderate fraction. Because the lost demands of
abandoning customers compensate for the excess of customer arrivals over the
service capacity, a call center operated in the ED regime can still achieve
reasonable service quality. More specifically, the mean customer waiting time is
comparable to the mean service time, a moderate fraction of customers abandon
the system, and all agents are almost always busy.

Service requirements for an outsourced call center are specified in the
service-level agreement (SLA) between the firm and the call center provider. The
SLA includes performance objectives such as the average service time, the
acceptable abandonment rate, and the acceptable customer delay time. One
important service level objective is a specified percentage of customers to be
served within a given delay, e.g., ``80\% of calls should be answered within one
minute.'' Based on this service level, \citet{BaronMilner09} studied SLA design
using the $ \mbox{M}/\mbox{M}/n+\mbox{M} $ model, and \citet{MandelbaumZeltyn09}
studied call center staffing using the $ \mbox{M}/\mbox{M}/n+\mbox{GI} $ model.
Despite the wide use of the above models, it was pointed out by
\citet{BrownETAL05} that the exponential service time distribution is not a
realistic assumption for call center customers. Therefore, these models may not
be able to provide adequate estimates for the performance measures required by
the SLA. Call center managers may need a more accurate but still tractable model
for performance analysis and staff deployment.

For queues in the ED regime, a fluid model proposed by \citet{Whitt06} is useful
for estimating several performance measures, including the fraction of
abandoning customers, the mean queue length, and the mean virtual waiting time.
In the $\mbox{M}/\mbox{M}/n+\mbox{GI}$ setting, the accuracy of the fluid model
was studied by \citet{BassambooRandhawa10}. They proved that in the steady
state, the accuracy gaps of fluid approximations for the mean queue length and
the rate of customer abandonment do not increase with the arrival rate. As a
deterministic model, however, the fluid model cannot be used to predict the
percentage of customers to be served within a time limit. A refined model is
thus necessary for estimating such a measure.

The focus of this paper is a diffusion model for many-server queues in the ED
regime. Both the service and patience time distributions are assumed to be
\emph{general}. Using this diffusion model, we obtain the steady-state
distributions of virtual waiting time and queue length, which in turn yield
approximate formulas for performance measures such as the service level
mentioned above. These formulas are able to produce accurate estimates,
especially when customer patience times are relatively long compared with their
service times. Empirical studies suggest that this requirement is realistic for
service-oriented call centers. For example, it was reported by
\citet{MandelbaumETAL01} and \citet{MandelbaumZeltyn13} that in the call center
of an Israeli bank, the mean customer patience time was several times longer
than the mean service time. By numerical experiments, we demonstrate that for
practical purposes, this diffusion model outperforms existing models that
rely on the exponential service time assumption; see Section~\ref{sec:staffing}.

The general service time assumption is a major challenge in the analysis of
many-server queues. In the literature, the studies of many-server queues with a
general service time distribution usually involve the analysis of one or several
infinite-dimensional processes, which are used for tracking customer age or
residual times. The resulting approximate models are also infinite-dimensional,
typically in the form of two-parameter or measure-valued processes; see, e.g.,
\citet{Whitt06}, \citet{KangRamanan10}, \citet{KaspiRamanan11}, \citet{Zhang13},
and \citet{KaspiRamanan13}. Because these approximate models are either
deterministic or too complex to be used for estimating a distribution, explicit
formulas have been absent from the literature for the steady-state virtual
waiting time and queue length distributions. Estimation of these distributions
relies heavily on simulation; see, e.g., \citet{BlanchetLam14}. The performance
formulas provided in this paper can fill this gap for queues in the ED regime.

We use a \emph{one-dimensional} Ornstein--Uhlenbeck (OU) process to approximate
the virtual waiting time process. We obtain this diffusion model by scaling a
many-server queue in both \emph{space} and \emph{time}, and then replacing the
centered and scaled arrival, service completion, and abandonment processes with
mutually independent Brownian motions. Depending on the service time
distribution through the first two moments, this model allows us to obtain
Gaussian approximations for the steady-state virtual waiting time and queue
length distributions. In contrast, the approximate models in the literature are
derived by scaling many-server queues in \emph{space} only. As a result, the
general service time distribution needs to be incorporated in the approximate
model, leading to an infinite-dimensional Markovian representation. It is well
known that when a many-server queue is critically loaded, the system performance
depends on the entire service time distribution and differs from that of a queue
with one or several servers significantly; see, e.g., \citet{DaiHeTezcan10},
\citet{MandelbaumMomcilovic12}, and \citet{KaspiRamanan13}. From a macroscopic
perspective in both space and time, we demonstrate that in the ED regime, the
dynamics of a many-server queue could be as simple as that of a single-server
queue, even though the service time distribution is assumed to be general. A
one-dimensional diffusion process, depending on the service time distribution
only by its first two moments, may suffice to capture the dynamics of the
many-server system. (One may also refer to \citet{He13}, an earlier version of
this paper, where we proved a common diffusion limit for systems with an
exponential patience time distribution, with either a single or many servers in
the overloaded regime, using a different method than the one in this paper.)
Such a simple model is much more attractive for analysis and control purposes.

The diffusion model and the performance formulas are rooted in the limit
theorems presented in Section~\ref{sec:Limits}. For queues in the ED regime,
\citet{DaiHeTezcan10} proved a multi-dimensional diffusion limit for the $
\mbox{GI}/\mbox{Ph}/n+\mbox{M} $ model, and \citet{HuangETAL14} proved a
one-dimensional diffusion limit for the $ \mbox{GI}/\mbox{M}/n+\mbox{GI} $
model. In addition, \citet{HuangETAL14} applied the obtained diffusion limit to
delay announcement in call centers. In this paper, we consider a sequence of
$\mbox{GI}/\mbox{GI}/n+\mbox{GI}$ queues indexed by the number of servers $n$,
and assume that the mean patience time goes to infinity as $n$ goes large. In
this asymptotic framework, the virtual waiting time process of each queue is
scaled in both space and time, with the number of servers and the mean patience
time being the respective scaling factors. The joint scaling scheme is essential
to obtain a Brownian approximation for the service completion process,
eventually leading to a one-dimensional diffusion limit. In the previous studies
such as \citet{DaiHeTezcan10}, \citet{MandelbaumMomcilovic12}, and
\citet{KaspiRamanan13}, the queueing processes are scaled by the number of
servers only. In this case, only when the service time distribution is
exponential, will the scaled queueing process converges to a one-dimensional
diffusion process. To the best of our knowledge,
Theorem~\ref{theorem:many-server-virtual} in this paper is the first rigorous
result that identifies a one-dimensional diffusion limit for many-server queues
with a \emph{general} service time assumption.

The technique of joint scaling in space and time was adopted by
\citet{Whitt03,Whitt04}, \citet{Gurvich04}, and \citet{Atar12} for many-server
queues with an \emph{exponential} service time distribution. \citet{Whitt04}
considered the $ \mbox{M}/\mbox{M}/n/r+\mbox{M} $ model and proved that in the
ED regime, the queue length process has a diffusion limit when the product of
the number of servers and the mean patience time goes to infinity. A critically
loaded regime, known as the nondegenerate slowdown regime, was studied by
\citet{Whitt03}, \citet{Gurvich04}, and \citet{Atar12}. In this regime, the
diffusion limit for the queue length process is proved to be either a reflected
OU process when the patience time distribution is exponential, or a reflected
Brownian motion when there is no customer abandonment. Because of the
exponential service time assumption, it is certain that those diffusion limits
are one-dimensional. The ability of space-time scaling to simplify the analysis
of many-server queues is barely manifested by those papers. In this sense, it is
the \emph{general} service time assumption that distinguishes our work from
others. By means of space-time scaling, we provide a paradigm for building a
tractable model for many-server queues with a general service time distribution.

Proving the diffusion limit requires new fundamental tools.
Theorem~\ref{theorem:FCLT} in this paper is a functional central limit theorem
(FCLT) for the superposition of renewal processes. We prove that the
superposition of $n$ independent, identically distributed (iid) stationary
renewal processes, after being centered and scaled in space and time, converges
in distribution to a Brownian motion as $n$ goes large. For a many-server system
in heavy traffic, the space-time scaled service completion process is
characterized by this theorem, which allows us to bypass the analysis of the
queue's infinite-dimensional state process. To apply this theorem, we consider a
sequence of perturbed systems that are asymptotically equivalent to the original
queues. We assume that servers in a perturbed system are always busy so that the
service completion process is the superposition of $n$ renewal processes. The
simplified dynamics of the perturbed system enable us to follow the procedure in
\citet{HuangETAL14}, proving the diffusion limit by a continuous mapping
approach.

We would summarize the contributions of this paper as follows. First, the
explicit formulas obtained from the diffusion model are practical tools for
performance estimation and staff deployment in efficiency-driven service
systems. In spite of their simple expressions, these formulas are superior to
the widely used formulas relying on the exponential service time assumption.
Second, we establish limit theorems to justify the diffusion model. By means of
space-time scaling, we prove a one-dimensional diffusion limit for many-server
queues with a general service time distribution. This joint scaling approach is
useful for building tractable models for queues in the ED regime. Third, we
prove an FCLT for the superposition of renewal processes. It is a fundamental
result that characterizes the service completion process of a many-server queue
in heavy traffic. This theorem enables us to bypass the infinite-dimensional
analysis in proving the diffusion limit.

The remainder of the paper is organized as follows. The diffusion model and the
performance formulas are introduced in Section~\ref{sec:Diffusion}. The limit
theorems for the diffusion model are presented in Section~\ref{sec:Limits}. We
examine and discuss the approximate formulas by numerical experiments in
Section~\ref{sec:Numerical}. Section~\ref{sec:Proof-virtual} is dedicated to the
proof of the diffusion limit for the virtual waiting time process, and
Section~\ref{sec:Proof-queue} presents the proof for the queue length limit. The
paper is concluded in Section~\ref{sec:conclusion}. We leave the proof of the
FCLT and the proofs of technical lemmas to the appendix.

Let us close this section with frequently used notation. The space of functions
$f:\mathbb{R}_{+}\rightarrow\mathbb{R}^{k}$ that are right-continuous on
$[0,\infty)$ and have left limits on $(0,\infty)$ is denoted by $\mathbb{D}^{k}$
(with  $ \mathbb{D} = \mathbb{D}^{1} $), which is endowed with the Skorokhod
$J_{1}$ topology (see, e.g., \citet{Billingsley99}). For $ f\in\mathbb{D} $, we
use $ f(t-) $ to denote the left limit of $ f $ at $t$ and $\Delta f(t)$ the
increment at $t$, i.e., $\Delta f(t) = f(t)-f(t-)$. For $ T>0 $, we use $
\int_{0}^{T} |\mathrm{d}f(t)|$ to denote the total variation of $ f $ over $
[0,T] $. For $f^{\prime}\in\mathbb{D}$ that is nondecreasing and takes values in
$\mathbb{R}_{+}$, $f\circ f^{\prime}$ is the composed function in $\mathbb{D}$,
i.e., $(f\circ f^{\prime})(t) = f(f^{\prime}(t))$. We use $e$ for the identity
function on $\mathbb{R}_{+}$ and $\chi$ the constant one function on
$\mathbb{R}_{+}$, i.e., $e(t)=t$ and $\chi(t)=1$ for $t\geq0$.

\section{Diffusion model and performance formulas}
\label{sec:Diffusion}

Consider a $\mbox{GI}/\mbox{GI}/n+\mbox{GI}$ queue, whose customer arrival
process is a renewal process and service times are iid nonnegative random
variables. Customers are served by $n$ identical servers. Upon arrival, a
customer gets into service if an idle server is available; otherwise, he waits
in a buffer with infinite room. Waiting customers are served on the first-come,
first-served basis, and the servers are not allowed to idle if there are
customers waiting. Each customer has a random patience time. When a customer's
waiting time exceeds his patience time, the customer abandons the system without
being served. The patience times are iid nonnegative random variables, and the
sequences of interarrival, service, and patience times are mutually independent.

Let $\lambda$ be the customer arrival rate and $\mu$ the service rate of each
server. The traffic intensity satisfies $\rho = \lambda/(n\mu)>1$. When a
many-server queue becomes overloaded, all servers will be almost always busy, so
the fraction of abandoning customers will be around
\begin{equation}
\alpha = \frac{\rho-1}{\rho}.
\label{eq:abd-fraction}
\end{equation}
Assume that both interarrival times and service times have finite variances,
with squared coefficients of variation $c_{a}^{2}$ and $c_{s}^{2}$,
respectively. Let $\Theta$ be the distribution function of patience times and $
\gamma $ the mean patience time. Assume that $\Theta$ is absolutely continuous
with a bounded, strictly positive density function $ f_{\Theta} $.

Suppose that at time $t\geq0$, a hypothetical customer with infinite patience
arrives at the queue. Let $W(t)$ be the \emph{virtual waiting time} at $ t $,
i.e., the amount of time this hypothetical customer has to wait before getting
into service. As the queue comes into the steady state, the virtual waiting
time process fluctuates around a mean level $w$ that can be determined as
follows: Because $\Theta(w)$ is the fraction of customers whose patience times
are less than $w$, it should be approximately equal to the fraction of
abandoning customers. Then, $\Theta(w) = (\rho-1)/\rho$, which yields
\begin{equation}
w = \Theta^{-1}\Big(\frac{\rho-1}{\rho}\Big).
\label{eq:mean-virtual}
\end{equation}
To represent the fluctuation of the virtual waiting time process around this
equilibrium level, we introduce a centered and scaled version of $W$ by
\[
\tilde{W}(t) = \sqrt{\frac{n}{\gamma}}(W(\gamma t)-w).
\]
We refer to $ \tilde{W} $ as the \emph{diffusion-scaled} virtual waiting time
process. To obtain this process, we scale the virtual waiting time process in
both space and time after removing the mean $ w $. Besides the commonly used
scaling in space by the number of servers, we also change the time scale of the
process by using the \emph{mean patience time} as the scaling factor.

We approximate the diffusion-scaled virtual waiting time process by an OU
process, which is given by the following stochastic differential equation
\[
\hat{W}(t) = \tilde{W}(0) + \hat{M}(t) - \rho\gamma f_{\Theta}(w)\int_{0}^{t}
\hat{W}(u)\,\mathrm{d}u \quad \mbox{for }t\geq 0.
\]
Here, $ \hat{M} $ is a driftless Brownian motion with $\hat{M}(0) = 0$ and
variance
\[
\hat{\sigma}_{m}^{2} = \frac{c_{a}^{2}+\rho c_{s}^{2}+\rho-1}{\rho\mu}.
\]
The OU process is a reasonable approximate model because the virtual waiting
time process is mean-reverting: As the virtual waiting time fluctuates around $
w $, the instantaneous abandonment rate from the buffer fluctuates accordingly.
If the probability density function of patience times does not change much
around $ w $, the relative variation in the abandonment rate will be
approximately proportional to the relative variation in the virtual waiting
time. When the virtual waiting time is either too long or too short, the
increased or decreased abandonment rate will pull it back toward the mean level
$ w $.

For the diffusion approximation to be accurate, the mean patience time $\gamma$,
serving as the scaling factor in time, should be \emph{relatively long} compared
with the mean service time. This requirement can be justified by
Theorem~\ref{theorem:many-server-virtual}, where the ED regime is formulated
into an asymptotic framework and $ \hat{W} $ is proved to be the limit of a
sequence of diffusion-scaled virtual waiting time processes. Although the mean
patience time goes to infinity in this asymptotic framework, the diffusion model
may still produce accurate performance estimates when it is just several times
longer than the mean service time. We will discuss the influence of the mean
patience time in Section~\ref{sec:influence}.

The OU process is strongly ergodic with a Gaussian steady-state distribution
(see, e.g., \citet{KarlinTaylor81}). More specifically, the steady-state
distribution of $\hat{W}$ has mean zero and variance
\[
\hat{\sigma}_{w}^{2} = \frac{c_{a}^{2}+\rho c_{s}^{2}+\rho-1} {2\rho^{2}
\mu\gamma f_{\Theta}(w)}.
\]
Let $ W(\infty) $ be the virtual waiting time in the steady state and $
\tilde{W}(\infty) $ be the diffusion-scaled version. Because $\hat{W}$ is an
approximation of $\tilde{W}$, their steady-state distributions are expected to
be close, i.e.,
\begin{equation}
\mathbb{P}[\tilde{W}(\infty)>a]\approx
1-\Phi\Big(\frac{a}{\hat{\sigma}_{w}}\Big)
\quad\mbox{for }a\in\mathbb{R},
\label{eq:W-tilde-tail}
\end{equation}
where $\Phi$ is the standard Gaussian distribution function. Consequently, the
steady-state virtual waiting time approximately follows a Gaussian distribution
with mean $w$ and variance
\begin{equation}
\sigma_{w}^2= \frac{c_{a}^{2}+\rho c_{s}^{2}+\rho-1}{2 n\rho^{2}
\mu f_{\Theta}(w)}.
\label{eq:variance-virtual}
\end{equation}

The percentage of customers to be served within a specified delay is referred to
as the \emph{service level} in practice. Let $ \zeta $ be a random variable that
has distribution $ \Theta $ and is independent of $ W(\infty) $. The service
level within $ d\geq 0 $ can be approximated by
\begin{equation}
\mathbb{P}[W(\infty) \leq \zeta \wedge d] \approx
\int_{0}^{\infty}\Phi_{w}(u\wedge d)f_{\Theta}(u)\,\mathrm{d}u
\label{eq:percentage}
\end{equation}
where
\[
\Phi_{w}(u)=\Phi\bigg(  \frac{\rho(u-w)\sqrt{2n\mu f_{\Theta}(w)}}{\sqrt
	{c_{a}^{2}+\rho c_{s}^{2}+\rho-1}}\bigg)  .
\]
In many service systems, the fraction of abandoning customers whose actual
waiting times exceed a short delay is an important measure of the quality of
service and customer satisfaction. We refer to this fraction as the
\emph{effective abandonment fraction} because it excludes those abandoning
customers whose effort of waiting is insignificant. For queues in the ED regime,
the fraction of abandoning customers out of those whose waiting times exceed $
d\geq 0 $ can be approximated by
\begin{equation}
\mathbb{P}[\zeta \leq W(\infty)\,|\, \zeta\wedge W(\infty) > d] \approx
\frac{\int_{d}^{\infty} (1-\Phi_{w}(u)) f_{\Theta}(u) \,\mathrm{d}u}
{(1-\Theta(d))(1-\Phi_{w}(d))}.
\label{eq:fraction-abandonment}
\end{equation}
Note that when $ d\geq w $, we cannot use the fluid model by \citet{Whitt06} to
estimate this fraction. This is because the steady-state virtual waiting time is
equal to $ w $ in the fluid model, by which the estimate of $ \mathbb{P}[\zeta
\wedge W(\infty) > d] $ must be zero for $ d\geq w $.

We are also interested in the distribution of the steady-state queue length. For
$ 0 < u < w $, the probability that a customer who arrived $ u $ time units ago
is still waiting in the buffer is around $ 1-\Theta(u) $. The mean queue length
(i.e., the mean number of customers in the buffer) can thus be approximated by
\begin{equation}
q = \int_{0}^{w} \lambda(1-\Theta(u))\,\mathrm{d}u.
\label{eq:mean-queue}
\end{equation}
Let $X(t)$ be the number of customers in the system at time~$t$, which
fluctuates around $n+q$ as the queue comes into the steady state. A centered and
scaled version of $ X $ is defined by
\[
\tilde{X}(t)=\frac{1}{\sqrt{n\gamma}}(X(\gamma t)-n-q).
\]
We refer to $ \tilde{X} $ as the \emph{diffusion-scaled} queue length process.

Let $X(\infty)$ be the number of customers in the steady state and $
\tilde{X}(\infty) $ be the diffusion-scaled version.
Theorem~\ref{theorem:many-server} in Section~\ref{sec:Limits} implies that
$\tilde{X}(\infty)$ approximately follows a Gaussian distribution with mean zero
and variance
\[
\hat{\sigma}^{2}_{x} =
\frac{\mu(c_{a}^{2}+\rho c_{s}^{2}+\rho-1)} {2\rho^{2}
	\gamma f_{\Theta}(w)}
+ \frac{\rho\mu}{\gamma}\int_{0}^{w} \Theta(u)(1-\Theta(u))\,\mathrm{d}u
+ \frac{\rho\mu c_{a}^{2}}{\gamma} \int_{0}^{w} (1-\Theta(u))^{2}\,\mathrm{d}u,
\]
i.e.,
\begin{equation}
\mathbb{P}[\tilde{X}(\infty)>a] \approx 1-\Phi\Big(\frac{a}{\hat{\sigma}_{x}}\Big)\quad\mbox{for }a\in\mathbb{R}.
\label{eq:X-tilde-tail}
\end{equation}
Hence, the steady-state number of customers approximately follows a Gaussian
distribution with mean $n+q$ and variance
\begin{equation}
\sigma^{2}_{x} =
\frac{n\mu(c_{a}^{2}+\rho c_{s}^{2}+\rho-1)} {2\rho^{2} f_{\Theta}(w)}
+ n\rho \mu\int_{0}^{w} \Theta(u)(1-\Theta(u))\,\mathrm{d}u
+ n\rho \mu c_{a}^{2} \int_{0}^{w} (1-\Theta(u))^{2}\,\mathrm{d}u.
\label{eq:variance-queue}
\end{equation}

\section{Limit theorems}
\label{sec:Limits}

We present the underlying theorems of the diffusion model in this section. To
formulate the ED regime, let us consider a sequence of $\mbox{G}/ \mbox{GI}/ n +
\mbox{GI}$ queues indexed by the number of servers. We do not require the
arrival processes to be renewal, but simultaneous arrival of two or more
customers is not allowed. In each queue, the number of initial customers, the
arrival process, the sequence of service times, and the sequence of patience
times are mutually independent. All these queues have the same traffic intensity
$\rho>1$ and the same service time distribution. Since the service rate $\mu$ is
identical in all systems, the arrival rate of the $n$th system is $
\lambda_{n}=n\rho\mu $. Assume that the mean patience time goes to infinity as
$n$ goes large, i.e.,
\begin{equation}
\gamma_{n}\rightarrow\infty\quad\mbox{as }n\rightarrow\infty.
\label{eq:gamma_n}
\end{equation}
We do \emph{not} require any assumption on the increasing rate of $ \gamma_{n} $
towards infinity. Because the patience time distribution changes with $ n $, it
is necessary to define a normalized patience time distribution for all queues.
Let $ H $ be the distribution function of a nonnegative random variable with
mean one, i.e., $ \int_{0}^{\infty} u\,\mathrm{d}H(u) = 1 $. Assume that $ H $
is absolutely continuous with a bounded, strictly positive density function $
f_{H} $, i.e., there exists some $ \kappa > 0 $ such that
\begin{equation}
0 < f_{H}(u) < \kappa \quad \mbox{for }u\geq 0.
\label{eq:bounded-density}
\end{equation}
Using this normalized distribution, we define the distribution function of
patience times in the $ n $th system by
\begin{equation}
\Theta_{n}(u) = H\Big(\frac{u}{\gamma_{n}}\Big) \quad\mbox{for } u \geq 0.
\label{eq:Theta-n}
\end{equation}
Clearly, the mean patience time in the $ n $th system is equal to $ \gamma_{n}
$. By \eqref{eq:mean-virtual}, the mean virtual waiting time in the $ n $th
system is $ w_{n} = \Theta_{n}^{-1}((\rho-1)/\rho) $. We may thus define the
normalized mean virtual waiting time by
\[
\bar{w} = H^{-1}\Big( \frac{\rho-1}{\rho}\Big),
\]
which satisfies $ \bar{w} = w_{n}/\gamma_{n} $ for all $ n\in\mathbb{N} $.

Let $F$ be the distribution function of service times. We impose a mild
regularity condition on $F$, which is
\begin{equation}
\limsup_{u\downarrow 0}\frac{1}{u}(F(u)-F(0))<\infty,
\label{eq:conditionF}
\end{equation}
and assume that the third moment of $F$ is finite, i.e.,
\begin{equation}
\int_{0}^{\infty}u^{3}\,\mathrm{d}F(u)<\infty.
\label{eq:m3}
\end{equation}
Let $ F_{e} $ be the equilibrium distribution of $F$, given by
\[
F_{e}(t)=\mu\int_{0}^{t}(1-F(u))\,\mathrm{d}u\quad\mbox{for }t\geq0\mbox{.}
\]
We assign service times to customers according to the following procedure. Let
$\{\xi_{j,k}:j,k\in\mathbb{N}\}$ be a double sequence of independent nonnegative
random variables. For each $j\in\mathbb{N}$, $\xi_{j,1}$ follows distribution
$F_{e}$ and $\xi_{j,k}$ follows distribution $F$ for $k\geq2$. In the $n$th
system, assume that all $n$ servers are busy at time zero. For $j=1,\ldots,n$,
$\xi_{j,1}$ is assigned to the initial customer served by the $j$th server as
the residual service time at time zero. For $k\geq2$, $\xi_{j,k}$ is the service
time of the $k$th customer served by the $j$th server. By this assignment, for
all $j,k\in\mathbb{N}$, the $k$th service time by the $j$th server is identical
in all systems that have at least $j$ servers.

Let $E_{n}(t)$ be the number of arrivals in the $n$th system during time
interval $(0,t]$. Define the diffusion-scaled arrival process $\tilde{E}_{n}$ by
\[
\tilde{E}_{n}(t)=\frac{1}{\sqrt{n\gamma_{n}}}(E_{n}(\gamma_{n}t)-\lambda_{n}\gamma_{n}t).
\]
Let $N$ be a renewal process whose interrenewal times have mean one and variance
$c_{a}^{2}$. If $E_{n}$ is a renewal process with $E_{n}(t)=N(\lambda_{n}t)$, it
follows from the FCLT for renewal processes that
\begin{equation}
\tilde{E}_{n}\Rightarrow\hat{E}\quad\mbox{as }n\rightarrow\infty\mbox{,}
\label{eq:E}
\end{equation}
where $\hat{E}$ is a driftless Brownian motion with $\hat{E}(0)=0$ and variance
$\rho\mu c_{a}^{2}$. To allow for more general arrival processes, we take the
convergence in \eqref{eq:E} as an assumption rather than require each $E_{n}$ to
be a renewal process.

Let $W_{n}(t)$ be the virtual waiting time at $ t $ in the $n$th system, whose
diffusion-scaled version is
\[
\tilde{W}_{n}(t) =
\sqrt{\frac{n}{\gamma_{n}}}(W_{n}(\gamma_{n}t)-\gamma_{n}\bar{w}).
\]
Assume that there exists a random variable $ \hat{W}(0) $ such that
\begin{equation}
\tilde{W}_{n}(0)\Rightarrow \hat{W}(0)\quad\mbox{as }n\rightarrow\infty.
\label{eq:initial}
\end{equation}
The first theorem states the diffusion limit for the virtual waiting time
process in the ED regime.
\begin{theorem}
\label{theorem:many-server-virtual}
Assume that the sequence of $\mbox{G}/\mbox{GI}/n+\mbox{GI}$ queues described
above has a common traffic intensity $ \rho > 1 $ and satisfies conditions
\eqref{eq:gamma_n}--\eqref{eq:initial}. Then,
\[
\tilde{W}_{n}\Rightarrow\hat{W}\quad\mbox{as }n\rightarrow\infty,
\]
where $ \hat{W} $ is the OU process given by the following stochastic
differential equation
\begin{equation}
\hat{W}(t) = \hat{W}(0)+\hat{M}(t) - \rho f_{H}(\bar{w})\int_{0}^{t}
\hat{W}(u)\,\mathrm{d}u \quad \mbox{for }t\geq 0.
\label{eq:OU-limit}
\end{equation}
Here, $ \hat{M} $ is a driftless
Brownian motion with $ \hat{M}(0) = 0 $ and
variance $ (c_{a}^{2} + \rho c_{s}^{2} + \rho -1)/(\rho\mu)  $.
\end{theorem}

\begin{remark}
Let $ h $ be the hazard rate function of $ H $, i.e., $ h(u) = f_{H}(u)/(1-H(u))
$ for $ u\geq 0 $. Because $ H(\bar{w}) = (\rho-1)/\rho $, we have $
h(\bar{w}) = \rho f_{H}(\bar{w})$. The diffusion limit depends on the
interarrival and service time distributions through their first two moments, and
depends on the normalized patience time distribution through the hazard rate at
$ \bar{w} $. This is because by centering and space-time scaling, the
arrival and service completion processes are replaced by Brownian motions in
the limit process. Since the virtual waiting time process fluctuates around the
equilibrium level, the influence on the scaled abandonment process is mostly
dictated by the normalized patience hazard rate at $ \bar{w} $.
\end{remark}

The mean queue length in the $ n $th system can be computed by
\begin{equation}
q_{n} = \int_{0}^{w_{n}} \lambda_{n} (1-\Theta_{n}(u))\,\mathrm{d}u
= n\gamma_{n}\rho\mu \int_{0}^{\bar{w}} (1-H(u))\,\mathrm{d}u.
\label{eq:qn}
\end{equation}
Let $X_{n}(t)$ be the number of customers at time $t\geq0$ in the $n$th system.
Then, the diffusion-scaled queue length process is given by
\begin{equation}
\tilde{X}_{n}(t)=\frac{1}{\sqrt{n\gamma_{n}}}(X_{n}(\gamma_{n}t)-n-q_{n}).
\label{eq:tilde-xn}
\end{equation}
Put
\begin{equation}
\bar{W}_{n}(t) = \frac{1}{\gamma_{n}}W_{n}(\gamma_{n}t),
\label{eq:w-bar}
\end{equation}
which is the virtual waiting time at $ t $ in the $ n $th  time-scaled system.
For a given $ t\geq 0 $, the second theorem concerns the limit of the
diffusion-scaled queue length at time $ t+\bar{W}_{n}(t) $.
\begin{theorem}
\label{theorem:many-server}
Under the conditions of Theorem~\ref{theorem:many-server-virtual}, for any fixed
$ t \geq 0 $,
\[
\tilde{X}_{n}(t+\bar{W}_{n}(t)) \Rightarrow
\mu\hat{W}(t)+ \hat{G}(t) \quad\mbox{as }n\rightarrow\infty,
\]
where $\hat{W}$ is the OU process defined by \eqref{eq:OU-limit} and $
\hat{G}(t) $ is a Gaussian random variable with mean zero and variance
\[
\hat{\sigma}_{g}^{2} = \rho\mu \int_{0}^{\bar{w}} H(u)(1-H(u))\,\mathrm{d}u
+\rho\mu c_{a}^{2} \int_{0}^{\bar{w}} (1-H(u))^{2}\,\mathrm{d}u.
\]
In addition, $ \hat{W}(t) $ and $ \hat{G}(t) $ are mutually independent.
\end{theorem}

\begin{remark}
Let $ t $ go to infinity. Then, $\mu\hat{W}(t)+\hat{G}(t)$ converges in
distribution to a Gaussian random variable with mean zero and variance
\[
\hat{\sigma}^{2}_{x} = \frac{\mu(c_{a}^{2}+\rho c_{s}^{2}+\rho-1)} {2\rho^{2}
	f_{H}(\bar{w})} + \rho\mu\int_{0}^{\bar{w}} H(u)(1-H(u))\,\mathrm{d}u +
	\rho\mu
c_{a}^{2} \int_{0}^{\bar{w}} (1-H(u))^{2}\,\mathrm{d}u.
\]
When $ n $ is large, the distribution of $ \tilde{X}_{n}(\infty) $ should be
close to this Gaussian distribution, which leads to formula
\eqref{eq:X-tilde-tail}.
\end{remark}

The third theorem plays an essential role in proving the previous two theorems.
This theorem is an FCLT for the superposition of time-scaled, stationary renewal
processes, which are defined as follows. For $t\geq0$ and $j\in\mathbb{N}$, let
\[
N_{j}(t)=\max\{k\geq 0:\xi_{j,1}+\cdots+\xi_{j,k}\leq t\},
\]
where $\{\xi_{j,k}:j,k\in\mathbb{N}\}$ is the double sequence of random
variables defined earlier. If $\xi_{j,1} > t$, we take $N_{j}(t)=0$ by
convention. Because $ \xi_{j,1} $ follows distribution $ F_{e} $ and $ \xi_{j,k}
$ follows distribution $ F $ for $ k\geq 2 $, $\{N_{j}:j \in \mathbb{N}\}$ is a
sequence of iid stationary renewal processes.

\begin{theorem}
\label{theorem:FCLT}
Let $\{N_{j}:j\in\mathbb{N}\}$ be a sequence of iid stationary renewal
processes, i.e., the delay distribution $F_{e}$ of each renewal process is the
equilibrium distribution of the interrenewal distribution $F$. Assume that $F$,
having mean $1/\mu$ and squared coefficient of variation $ c_{s}^{2} $,
satisfies \eqref{eq:conditionF} and \eqref{eq:m3}. Let
\begin{equation}
B_{n}(t)=\sum_{j=1}^{n}N_{j}(t)
\label{eq:B}
\end{equation}
and $\{\gamma_{n}:n\in\mathbb{N}\}$ be a sequence of positive numbers such that
$\gamma_{n}\rightarrow\infty$ as $n\rightarrow\infty$. Then,
\[
\tilde{B}_{n}\Rightarrow\hat{B}\quad\mbox{as }n\rightarrow\infty\mbox{,}
\]
where
\begin{equation}
\tilde{B}_{n}(t)=\frac{1}{\sqrt{n\gamma_{n}}}(B_{n}(\gamma_{n}t)-n\mu\gamma_{n}t)
\label{eq:B_tilde}
\end{equation}
and $\hat{B}$ is a driftless Brownian motion with $\hat{B}(0)=0$ and variance
$\mu c_{s}^{2}$.
\end{theorem}

\begin{remark}
To better understand Theorem~\ref{theorem:FCLT}, let us compare this result with
two other FCLTs. By the FCLT for renewal processes, $\{(N_{1}(\ell t)-\ell\mu
t)/\sqrt{\ell}:t\geq 0\}$ converges in distribution to a Brownian motion as
$\ell$ goes to infinity; see, e.g., Theorem~5.11 in \citet{ChenYao01}. Clearly,
the increments of this time-scaled renewal process become independent of its
history as the scaling factor gets large. \citet{Whitt85} proved an FCLT for the
superposition of stationary renewal processes. It states that
$\{\sum_{j=1}^{n}(N_{j}(t)-\mu t)/\sqrt{n}:t\geq 0\}$ converges in distribution
to a zero-mean Gaussian process that has stationary increments and continuous
sample paths. In this FCLT, the superposition process is scaled in space only.
The covariance function of each stationary renewal process is retained in the
limit Gaussian process, which, in general, is \emph{not} a Brownian motion; see
Theorem~2 in \citet{Whitt85}. In our theorem, each superposition process is
scaled in both space and time. Squeezing the time scale erases the dependence of
the increments of $\tilde{B}_{n}$ to the history. The limit of these space-time
scaled superposition processes should thus be a Gaussian process with
independent, stationary increments and continuous sample paths, which must be a
Brownian motion.
\end{remark}

In the ED regime, all servers are nearly always busy, so the service completion
process is almost identical to the superposition of many renewal processes.
Theorem~\ref{theorem:FCLT} implies that the space-time scaled service completion
process can be approximated by a Brownian motion, which allows us to bypass the
analysis of the infinite-dimensional age or residual process in proving a limit
process. Hence, by zooming out our perspective in both space and time, we may
obtain a one-dimensional diffusion model for many-server queues with a general
service time distribution.

\section{Numerical experiments and discussion}
\label{sec:Numerical}

In this section, we examine the diffusion model by numerical experiments. We
first study the influence of the mean patience time on the accuracy of
approximation, and then use the diffusion model to solve a staffing problem.

\begin{table}[t]
\caption{Performance estimates for the $\mbox{M}/\mbox{GI}/100+\mbox{M}$ queue
with $\mu=1.0$ and $\rho=1.2$; simulation results (with $95\%$ confidence
intervals) are compared with approximate results (in \emph{italics}).}
\centering
\footnotesize{
\begin{tabular}{l|lllll}
\hline

\multicolumn{1}{c|}{} & \multicolumn{1}{c}{}
& \multicolumn{2}{c}{Virtual waiting time}
& \multicolumn{2} {c}{Queue length} \T \tabularnewline

\multicolumn{1}{l|}{Patience} & \multicolumn{1}{l}{Abd. fraction}
& \multicolumn{1}{l}{Mean} & \multicolumn{1}{l}{Variance}
& \multicolumn{1}{l}{Mean} & \multicolumn{1}{l}{Variance} \B \tabularnewline

\hline

\multicolumn{1}{c|}{} & \multicolumn{5}{c}{$\mbox{M}/\mbox{D}/100+\mbox{M}$}
\T \B \tabularnewline \cline{2-6}

$\gamma=1.0$
& $0.1668$  & $0.1851$  & $0.005322$ & $20.02$  & $73.11$  \T \tabularnewline
& $\pm0.000020$  & $\pm0.000028$  & $\pm0.0000030$ & $\pm0.0034$  & $\pm0.038$
\tabularnewline

& $\emph{0.1667}$  & $\emph{0.1823}$  & $\emph{0.005000}$ & $\emph{20.00}$
& $\emph{70.00}$  \B \tabularnewline

$\gamma=5.0$
& $0.1667$  & $0.9142$  & $0.02639$ & $99.99$  & $364.1$  \T \tabularnewline
& $\pm0.000021$  & $\pm0.00014$  & $\pm0.00042$ & $\pm0.017$  & $\pm4.3$
\tabularnewline

& $\emph{0.1667}$  & $\emph{0.9116}$  & $\emph{0.02500}$ & $\emph{100.0}$
& $\emph{350.0}$  \B \tabularnewline

$\gamma=10$
& $0.1667$  & $1.826$  & $0.05487$ & $200.0$  & $749.2$  \T \tabularnewline
& $\pm0.000021$  & $\pm0.00030$  & $\pm0.000086$ & $\pm0.035$  & $\pm1.2$
\tabularnewline

& $\emph{0.1667}$  & $\emph{1.823}$  & $\emph{0.05000}$ & $\emph{200.0}$
& $\emph{700.0}$  \B \tabularnewline

\hline

\multicolumn{1}{c|}{}& \multicolumn{5}{c}{$\mbox{M}/\mbox{E}_{2}/100+\mbox{M}$}
\T \B \tabularnewline \cline{2-6}

$\gamma=1.0$
& $0.1672$  & $0.1869$  & $0.007799$ & $20.07$  & $97.08$   \T \tabularnewline
& $\pm0.000040$  & $\pm0.000055$  & $\pm0.0000033$ & $\pm0.0062$  & $\pm0.041$
\tabularnewline

& $\emph{0.1667}$  & $\emph{0.1823}$  & $\emph{0.007500}$ & $\emph{20.00}$
& $\emph{95.00}$  \B \tabularnewline

$\gamma=5.0$
& $0.1666$  & $0.9152$  & $0.03812$ & $99.97$  & $481.0$  \T \tabularnewline
& $\pm0.000043$  & $\pm0.00031$  & $\pm0.000049$ & $\pm0.035$  & $\pm0.63$
\tabularnewline

& $\emph{0.1667}$  & $\emph{0.9116}$  & $\emph{0.03750}$ & $\emph{100.0}$
& $\emph{475.0}$  \B \tabularnewline

$\gamma=10$
& $0.1666$  & $1.827$  & $0.07567$ & $199.9$  & $956.4$  \T \tabularnewline
& $\pm0.000042$  & $\pm0.00058$  & $\pm0.00015$ & $\pm0.066$  & $\pm2.0$
\tabularnewline

& $\emph{0.1667}$  & $\emph{1.823}$  & $\emph{0.07500}$ & $\emph{200.0}$
& $\emph{950.0}$  \B \tabularnewline

\hline

\multicolumn{1}{c|}{} & \multicolumn{5}{c}{$\mbox{M}/\mbox{LN}/100+\mbox{M}$}
\T \B \tabularnewline \cline{2-6}

$\gamma=1.0$
& $0.1679$  & $0.1890$  & $0.01049$ & $20.14$  & $122.2$  \T \tabularnewline
& $\pm0.000040$  & $\pm0.000052$  & $\pm0.0000043$ & $\pm0.0055$  & $\pm0.049$
\tabularnewline

& $\emph{0.1667}$  & $\emph{0.1823}$  & $\emph{0.01500}$ & $\emph{20.00}$
& $\emph{170.0}$  \B \tabularnewline

$\gamma=5.0$
& $0.1666$  & $0.9178$  & $0.06474$ & $99.97$  & $745.5$  \T \tabularnewline
& $\pm0.000043$  & $\pm0.00027$  & $\pm0.000068$ & $\pm0.029$  & $\pm0.73$
\tabularnewline

& $\emph{0.1667}$  & $\emph{0.9116}$  & $\emph{0.07500}$ & $\emph{100.0}$
& $\emph{850.0}$  \tabularnewline

$\gamma=10$
& $0.1666$  & $1.829$  & $0.1365$ & $199.9$  & $1563$  \T \tabularnewline
& $\pm0.000042$  & $\pm0.00054$  & $\pm0.00019$ & $\pm0.057$  & $\pm1.9$
\tabularnewline

& $\emph{0.1667}$  & $\emph{1.823}$  & $\emph{0.1500}$ & $\emph{200.0}$
& $\emph{1700}$  \B \tabularnewline

\hline

\end{tabular}
}
\label{table:measure100}
\end{table}

\begin{table}[t]
\caption{Tail probabilities for the steady-state virtual waiting time and queue
length in the $\mbox{M}/\mbox{GI}/ 100+\mbox{M}$ queue with $\mu=1.0$ and
$\rho=1.2$; simulation results (with $95\%$ confidence intervals) are compared
with diffusion approximations (in \emph{italics}).}
\centering
\footnotesize{
\begin{tabular}{l|llllll}
	
\hline

& \multicolumn{3}{c}{$\mathbb{P}[\tilde{W}(\infty)>a]$}
& \multicolumn{3}{c}{$\mathbb{P}[\tilde{X}(\infty)>a]$} \T \tabularnewline

\multicolumn{1}{c|}{Patience}
& \multicolumn{1}{c}{$a=0.5$} & \multicolumn{1}{c}{$a=1.0$}
& \multicolumn{1}{c}{$a=2.0$}
& \multicolumn{1}{c}{$a=0.5$} & \multicolumn{1}{c}{$a=1.0$}
& \multicolumn{1}{c}{$a=2.0$} \B \tabularnewline

\hline

\multicolumn{1}{c|}{} & \multicolumn{6}{c}{$\mbox{M}/\mbox{D}/100+\mbox{M}$}
\T \B \tabularnewline \cline{2-7}

$\gamma=1.0$
& $0.2584$  & $0.09269$  & $0.003869$  & $0.2559$  & $0.1131$  & $0.01140$
\T \tabularnewline
& $\pm0.00014$  & $\pm0.000078$  & $\pm0.000018$
& $\pm0.00014$  & $\pm0.000089$  & $\pm0.000031$  \B \tabularnewline

$\gamma=5.0$
& $0.2505$  & $0.08689$  & $0.003138$ & $0.2707$  & $0.1200$ & $0.01120$
\T \tabularnewline
& $\pm0.0019$  & $\pm0.0016$  & $\pm0.00017$
& $\pm0.0013$  & $\pm0.0013$  & $\pm0.00031$  \B \tabularnewline

$\gamma=10$
& $0.2539$  & $0.09004$  & $0.003419$ & $0.2840$  & $0.1252$ & $0.01089$
\T \tabularnewline
& $\pm0.00046$  & $\pm0.00023$  & $\pm0.000050$
& $\pm0.00049$  & $\pm0.00029$  & $\pm0.000093$  \B \tabularnewline

& $\emph{0.2398}$  & $\emph{0.07865}$ & $\emph{0.002339}$
& $\emph{0.2750}$  & $\emph{0.1160}$  & $\emph{0.008414}$  \T \B \tabularnewline

\hline

\multicolumn{1}{c|}{} &\multicolumn{6}{c}{$\mbox{M}/\mbox{E}_{2}/100+\mbox{M}$}
\T \B \tabularnewline \cline{2-7}

$\gamma=1.0$
& $0.3007$  & $0.1422$  & $0.01596$ & $0.2865$  & $0.1472$  & $0.02314$
\T \tabularnewline
& $\pm0.00023$  & $\pm0.00015$  & $\pm0.000039$
& $\pm0.00023$  & $\pm0.00015$  & $\pm0.000044$  \B \tabularnewline

$\gamma=5.0$
& $0.2884$  & $0.1302$  & $0.01215$ & $0.2972$  & $0.1523$ & $0.02261$
\T \tabularnewline
& $\pm0.00055$  & $\pm0.00039$  & $\pm0.000098$
& $\pm0.00056$  & $\pm0.00041$  & $\pm0.00014$  \B \tabularnewline

$\gamma=10$
& $0.2859$  & $0.1279$  & $0.01151$ & $0.3057$  & $0.1538$ & $0.02095$
\T \tabularnewline
& $\pm0.00073$  & $\pm0.00054$  & $\pm0.00014$
& $\pm0.00074$  & $\pm0.00059$  & $\pm0.00021$  \B \tabularnewline

& $\emph{0.2819}$  & $\emph{0.1241}$  & $\emph{0.01046}$
& $\emph{0.3040}$  & $\emph{0.1525}$  & $\emph{0.02009}$  \T \B \tabularnewline

\hline

\multicolumn{1}{c|}{} & \multicolumn{6}{c}{$\mbox{M}/\mbox{LN}/100+\mbox{M}$}
\T \B \tabularnewline \cline{2-7}

$\gamma=1.0$
& $0.3288$  & $0.1826$  & $0.03577$ & $0.3099$  & $0.1774$  & $0.03847$
\T \tabularnewline
& $\pm0.00019$  & $\pm0.00014$  & $\pm0.000048$
& $\pm0.00017$  & $\pm0.00013$  & $\pm0.000048$   \B \tabularnewline

$\gamma=5.0$
& $0.3348$  & $0.1952$  & $0.04389$ & $0.3343$  & $0.2040$ & $0.05275$
\T \tabularnewline
& $\pm0.00036$  & $\pm0.00026$  & $\pm0.00016$
& $\pm0.00035$  & $\pm0.00026$  & $\pm0.00016$  \B \tabularnewline

$\gamma=10$
& $0.3371$  & $0.1997$  & $0.04703$ & $0.3452$  & $0.2118$ & $0.05492$
\T \tabularnewline
& $\pm0.00049$  & $\pm0.00035$  & $\pm0.00024$
& $\pm0.00049$  & $\pm0.00040$  & $\pm0.00024$  \B \tabularnewline

& $\emph{0.3415}$  & $\emph{0.2071}$  & $\emph{0.05124}$
& $\emph{0.3507}$  & $\emph{0.2216}$  & $\emph{0.06252}$  \T \B \tabularnewline

\hline

\end{tabular}
}
\label{table:tail100}
\end{table}

\subsection{Influence of the mean patience time}
\label{sec:influence}

Serving as the respective scaling factors in space and time, the number of
servers and the mean patience time will affect how close the queue's performance
is to the diffusion approximation. By Theorem~\ref{theorem:many-server-virtual},
both scaling factors are required to approach infinity in order for the
diffusion-scaled virtual waiting time process to converge. A diffusion model can
generally produce satisfactory performance approximations when there are at
least tens of servers (see, e.g., \citet{GarnettETAL02}, \citet{DaiHe13}, and
\citet{HuangETAL14}). However, it is not immediately clear how large the mean
patience time should be in order for the proposed diffusion model to be
sufficiently accurate. We would thus evaluate the influence of the mean patience
time on the accuracy of approximation.

Assume that the queue has a Poisson arrival process with rate $ \lambda=120 $
and 100 servers with service rate $ \mu=1.0 $, so the traffic intensity is $
\rho=1.2 $. The patience time distribution is exponential with mean $ \gamma =
1.0 $, $ 5.0 $, or $ 10 $. The service time distribution may be deterministic,
Erlang (with two stages), or log-normal, denoted by $\mbox{D}$, $\mbox{E}_2$,
and $\mbox{LN}$, respectively. With $c_{s}^{2}=0$ and $0.5$, the deterministic
and Erlang distributions are used to represent scenarios where service times
have small variability. \citet{BrownETAL05} reported that a log-normal
distribution provides a good fit for service time data from an Israeli call
center. We also test such a distribution with $c_{s}^{2}=2.0$, which yields more
variable service times.

The estimates of several performance measures, including the fraction of
abandoning customers, the mean and variance of the steady-state virtual waiting
time, and the mean and variance of the steady-state queue length, are listed in
Table~\ref{table:measure100}. We use \eqref{eq:abd-fraction},
\eqref{eq:mean-virtual}, \eqref{eq:variance-virtual}, \eqref{eq:mean-queue}, and
\eqref{eq:variance-queue} to compute the respective approximate results. The
formulas for the fraction of abandoning customers, the mean virtual waiting
time, and the mean queue length are identical to those obtained from the fluid
model by \citet{Whitt06}. These fluid approximations agree with the simulation
results very well, which is consistent with the conclusion drawn by
\citet{Whitt06}: The fluid model is able to produce accurate approximations for
mean performance measures in the ED regime.

The fluid model, however, cannot be used for estimating variances because of its
deterministic nature. We need the diffusion model for completing this task.
Theorems~\ref{theorem:many-server-virtual} and~\ref{theorem:many-server} imply
that diffusion approximations are more accurate when the mean patience time is
longer. Comparing the variance results in Table~\ref{table:measure100}, however,
we can see that an adequate diffusion approximation may \emph{not} require the
mean patience time to be large. With a mean patience time that is comparable to
the mean service time, the approximate variances are satisfactory when service
times are deterministic or follow an Erlang distribution. This observation can
be explained as follows. Because all servers are almost always busy, the service
completion process is close to the superposition of iid renewal processes. As we
discussed in the previous section, by squeezing the time scale, the increments
of the service completion process will become more and more independent of the
history. Then by Theorem~\ref{theorem:FCLT}, we may use a Brownian motion to
replace the space-time scaled service departure process to obtain the diffusion
model. If the variability in service times is \emph{not} significant, a
\emph{moderate} scaling factor in time is sufficient for the Brownian
approximation to work well. Therefore, with a deterministic or Erlang service
time distribution, the approximate variances are close to the simulation results
even for $\gamma=1.0$. A \emph{larger} scaling factor is necessary if the
variability in service times is \emph{more} considerable. When the service time
distribution is log-normal with $c_{s}^{2}=2.0$, the approximate variances are
no longer accurate for $\gamma=1.0$. In order for the increments of the
time-scaled service completion process to be sufficiently independent of the
history, the mean patience time should be at least \emph{several times longer}
than the mean service time. We can see that the approximate variances are
satisfactory for $\gamma=5.0$ and~$10$.

We also compute tail probabilities for the steady-state distributions of
diffusion-scaled virtual waiting time and queue length. Approximations by
\eqref{eq:W-tilde-tail} and \eqref{eq:X-tilde-tail} are compared with simulation
results in Table~\ref{table:tail100}, the observation from which is consistent
with what we found from Table~\ref{table:measure100}: With the deterministic or
Erlang service time distribution, the approximate distributions are satisfactory
when the mean patience time is comparable to or longer than the mean service
time; when service times have a larger variance, the mean patience time needs to
be at least several times longer than the mean service time, in order for
Gaussian approximations to be accurate.

\begin{figure}[t]
\centering
\begin{subfigure}[b]{0.49\textwidth}
\includegraphics[width=3.1in]{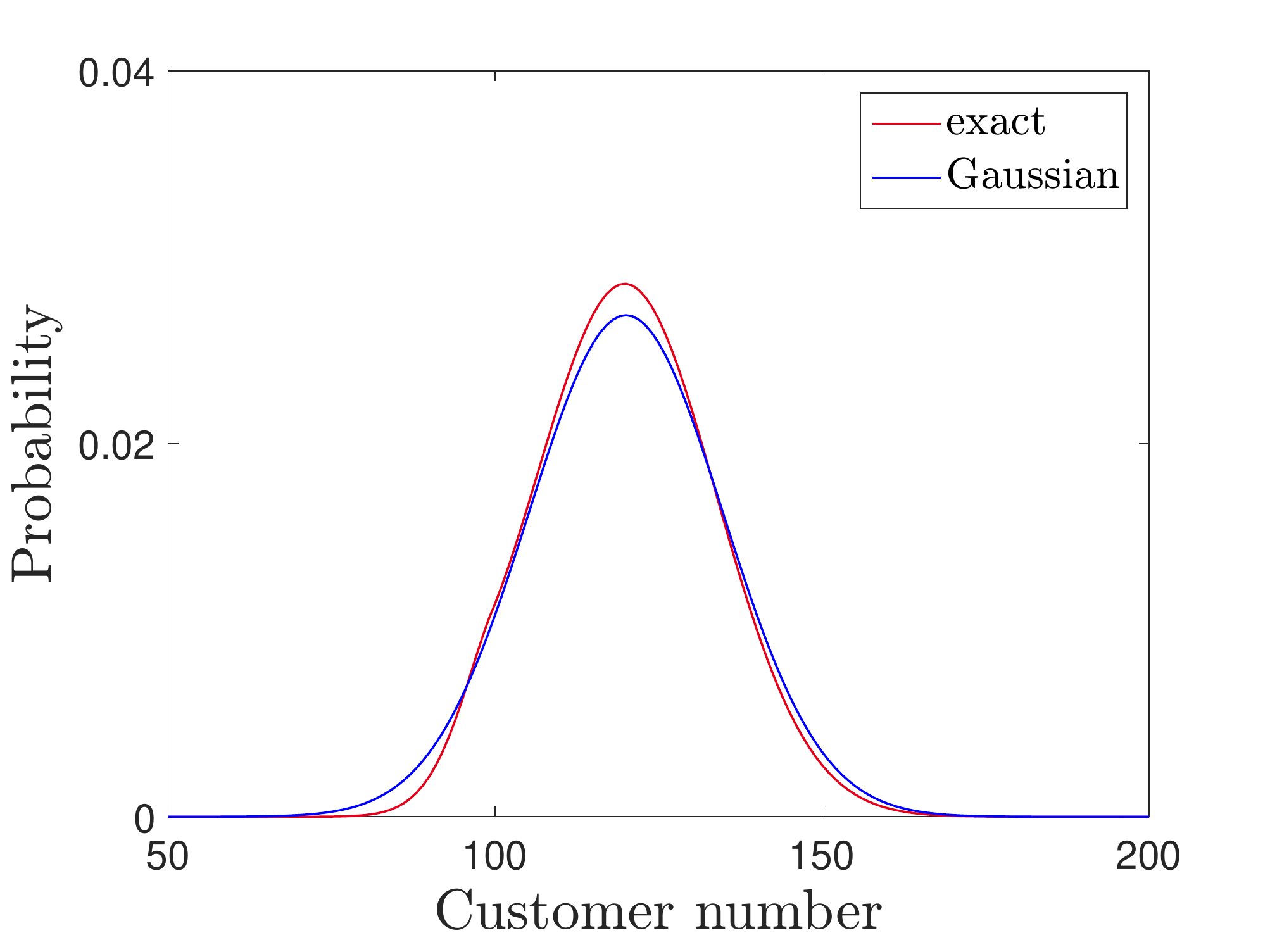}
\caption{$\gamma = 1.0$}
\label{fig:H2-gamma1}
\end{subfigure}
\begin{subfigure}[b]{0.49\textwidth}
\includegraphics[width=3.1in]{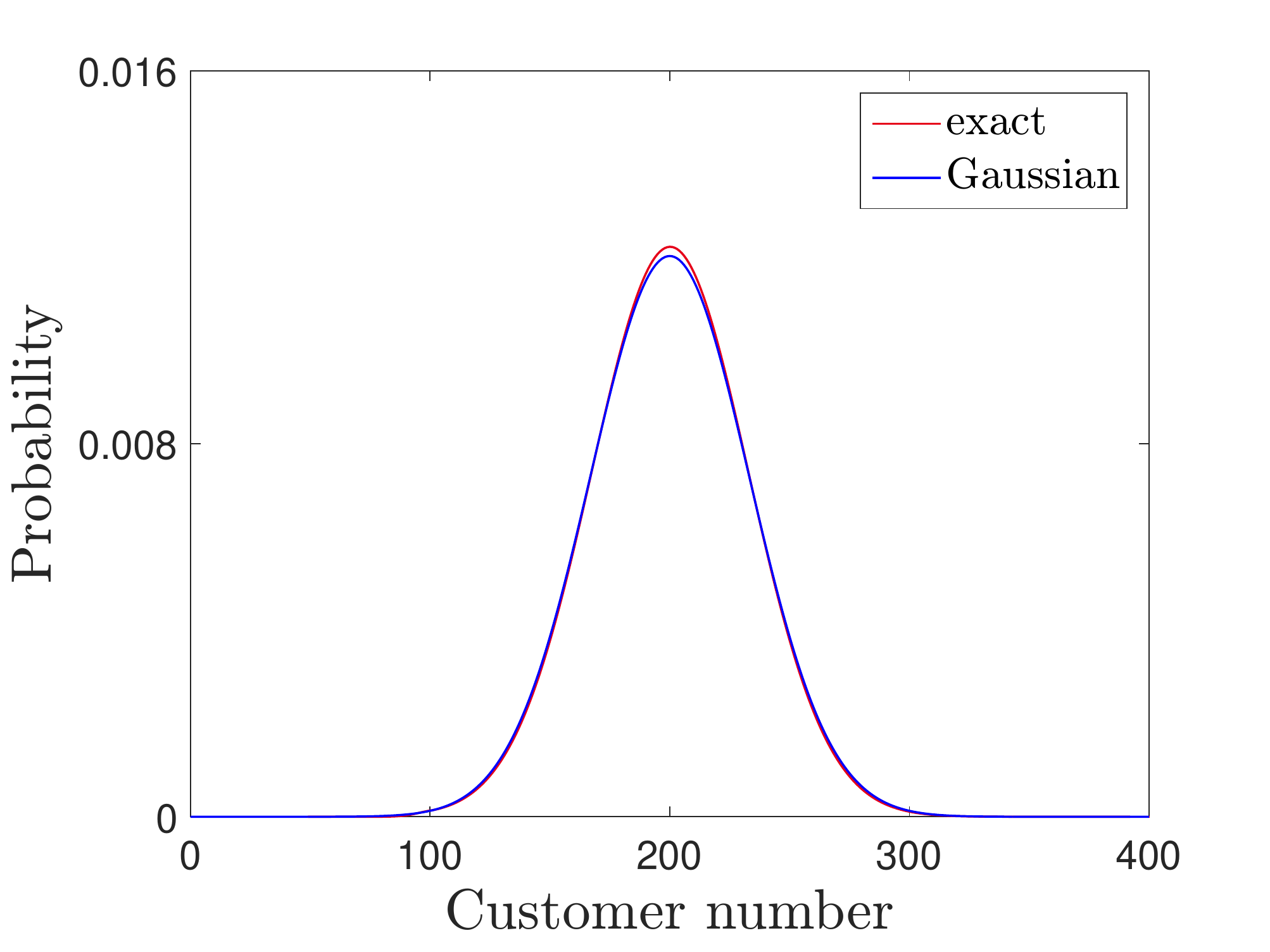}
\caption{$\gamma = 5.0$}
\label{fig:H2-gamma5}
\end{subfigure}
\caption{The steady-state distribution of the number of customers in the
$\mbox{M}/\mbox{H}_{2}/100+\mbox{M}$ queue with $\mu=1.0$, $\rho=1.2$, and
$c_{s}^{2}=3.0$; the exact distribution by the matrix-analytic method is
compared with Gaussian approximation in \eqref{eq:Gaussian-approx}.}
\label{fig:H2}
\end{figure}

To illustrate how the number of customers converges to a Gaussian random
variable in the steady state, we examine an $\mbox{M}/\mbox{H}_{2}/100+\mbox{M}$
queue whose service times follow a hyperexponential distribution with $\mu=1.0$
and $c_{s}^{2}=3.0$. There are two types of customers in the queue, and the
service times of either type are iid exponential random variables. The fraction
of the first type is $59.16\%$ and its mean service time is $0.1691$; the
fraction of the second type is $40.84\%$ and its mean service time is $2.203$.
We are interested in this queue because the exact distribution of the
steady-state number of customers can be computed by the matrix-analytic method
(see, e.g., \citet{LatoucheRamaswami99}). We may also approximate the
distribution of $ X(\infty) $ by
\begin{equation}
\mathbb{P}[X(\infty) = i] \approx \frac{1}{\sigma_{x}}
\phi\Big(\frac{i-n-q}{\sigma_{x}}\Big)
\quad\mbox{for }i=0,1,\ldots,
\label{eq:Gaussian-approx}
\end{equation}
where $ \phi $ is the standard Gaussian density function and $ q $ and $
\sigma_{x} $ are given by \eqref{eq:mean-queue} and \eqref{eq:variance-queue},
respectively. We compare these two distributions in Figure~\ref{fig:H2}.
Although the Gaussian approximation cannot capture the exact distribution for
$\gamma=1.0$, it becomes a good fit for $\gamma=5.0$.

Through the above numerical examples, we can tell that both the variability in
service times and the ratio of the mean patience time to the mean service time
have influence on the accuracy of approximation. We would thus introduce the
following quantity
\[
R_{\operatorname{svp}} = \frac{c_{s}}{\gamma\mu},
\]
or the ratio of the standard deviation of service times to the mean patience
time, as an index of accuracy when a queue is approximated by the diffusion
model. This quantity is called the \emph{service-variability-to-patience ratio
(SVPR)}. (Essentially, the SVPR is an index of how close the space-time scaled
service completion process is to a Brownian motion.) By extensive numerical
experiments, we observe that the diffusion approximations are generally accurate
when $ R_{\operatorname{svp}} < 0.5$.

\subsection{Staffing using the diffusion model}
\label{sec:staffing}

\begin{figure}[t]
\centering
\includegraphics[width=3.1in]{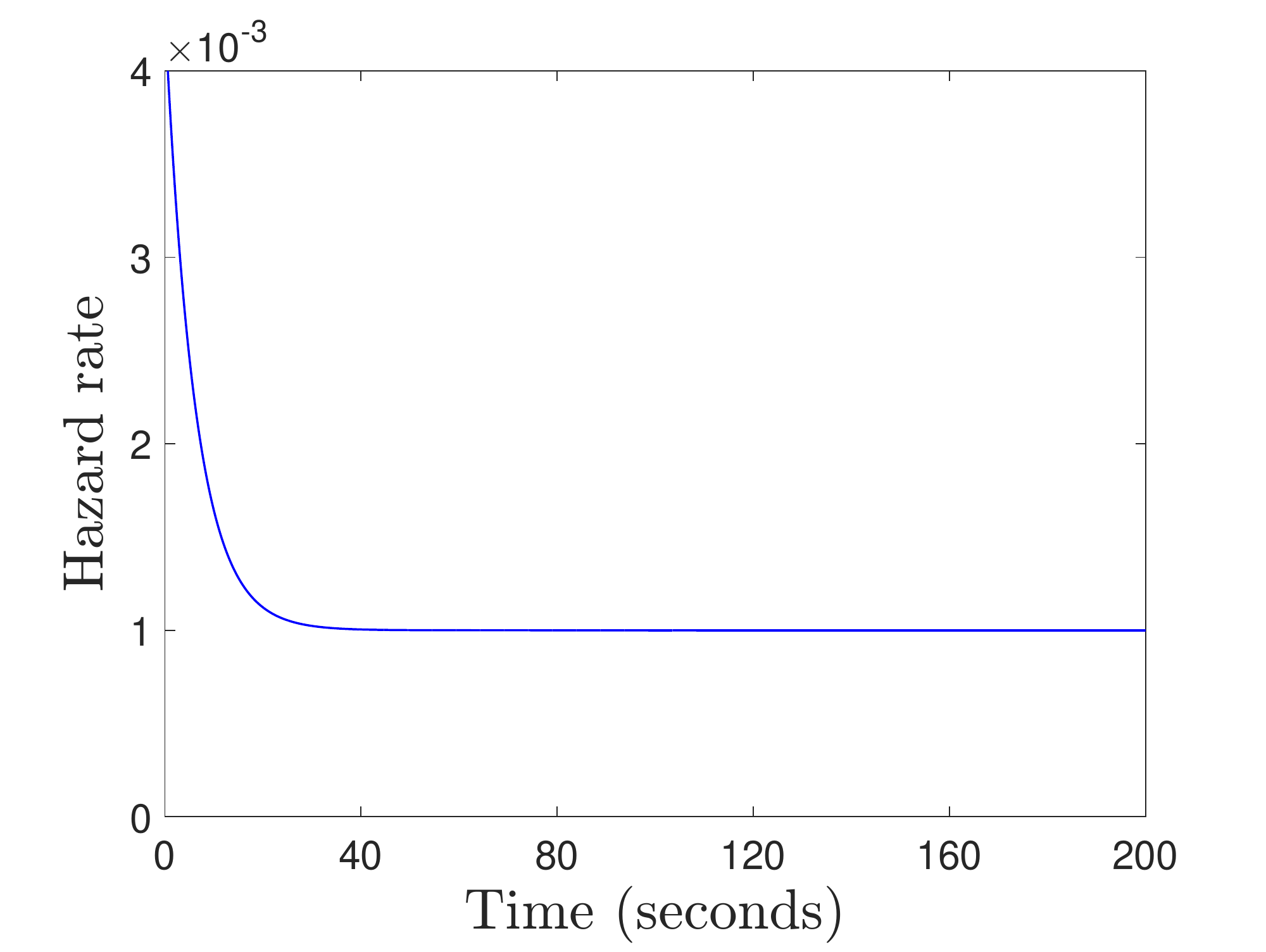}
\caption{The hazard rate of the hyperexponential distribution.}
\label{fig:H2-patience}
\end{figure}

In typical service-oriented call centers, the variability in customer service
times is not significant and most customers are relatively patient when they are
waiting for service. For example, by analyzing a set of operational data from an
Israeli call center, \citet{MandelbaumETAL01} reported that the customer service
times had moderate variability, ranging from $ c_{s}^{2} = 1.54$ to $ c_{s}^{2}
= 5.83 $ for different months, and that the mean patience time was about four
times longer than the mean service time. Both \emph{moderate service
variability} and \emph{long customer patience} suggest that the diffusion model
is appropriate for modeling call centers in the ED regime. Next, let us study an
application of the diffusion model to call center staffing.

Consider an $ \mbox{M}/\mbox{LN}/n+\mbox{H}_{2} $ queue, which has a Poisson
arrival process, a log-normal service time distribution, and a hyperexponential
patience time distribution. There are two types of customers with different
abandonment behaviors, and the patience times of either type are iid exponential
random variables. We assume that 98\% of customers have long patience times with
mean $ 1000 $ seconds and 2\% of customers have short patience times with mean $
6.0 $ seconds. We use such a distribution to represent a typical pattern of
abandonment in call centers: While most customers would wait patiently for their
service, a small fraction of customers would hang up within seconds if they
cannot be served immediately.  The hazard rate function of patience times is
plotted in Figure~\ref{fig:H2-patience}. We adjusted the parameters in order for
the hyperexponential distribution to imitate the patience time distribution in
the call center of a large U.S. bank, whose operational data were analyzed by
\citet{MandelbaumZeltyn13}. By checking Figure~2 in their paper, one can see
that our hazard rate is just a ``replica'' of the smoothed hazard rate of their
patience time data.

In this queue, the service rate is taken to be $ \mu = 1/230 $ per second, which
is obtained from the estimated mean service time in the U.S. call center; see
Figure~16 in \citet{MandelbaumZeltyn13}. Their paper, however, does not provide
information about the variance of service time. We would consider two scenarios
where the log-normal service time distribution has $ c_{s}^{2} = 3.0 $ and $
c_{s}^{2} = 5.0 $, respectively.

\begin{figure}[t]
\centering
\begin{subfigure}[b]{0.49\textwidth}
\includegraphics[width=3.1in]{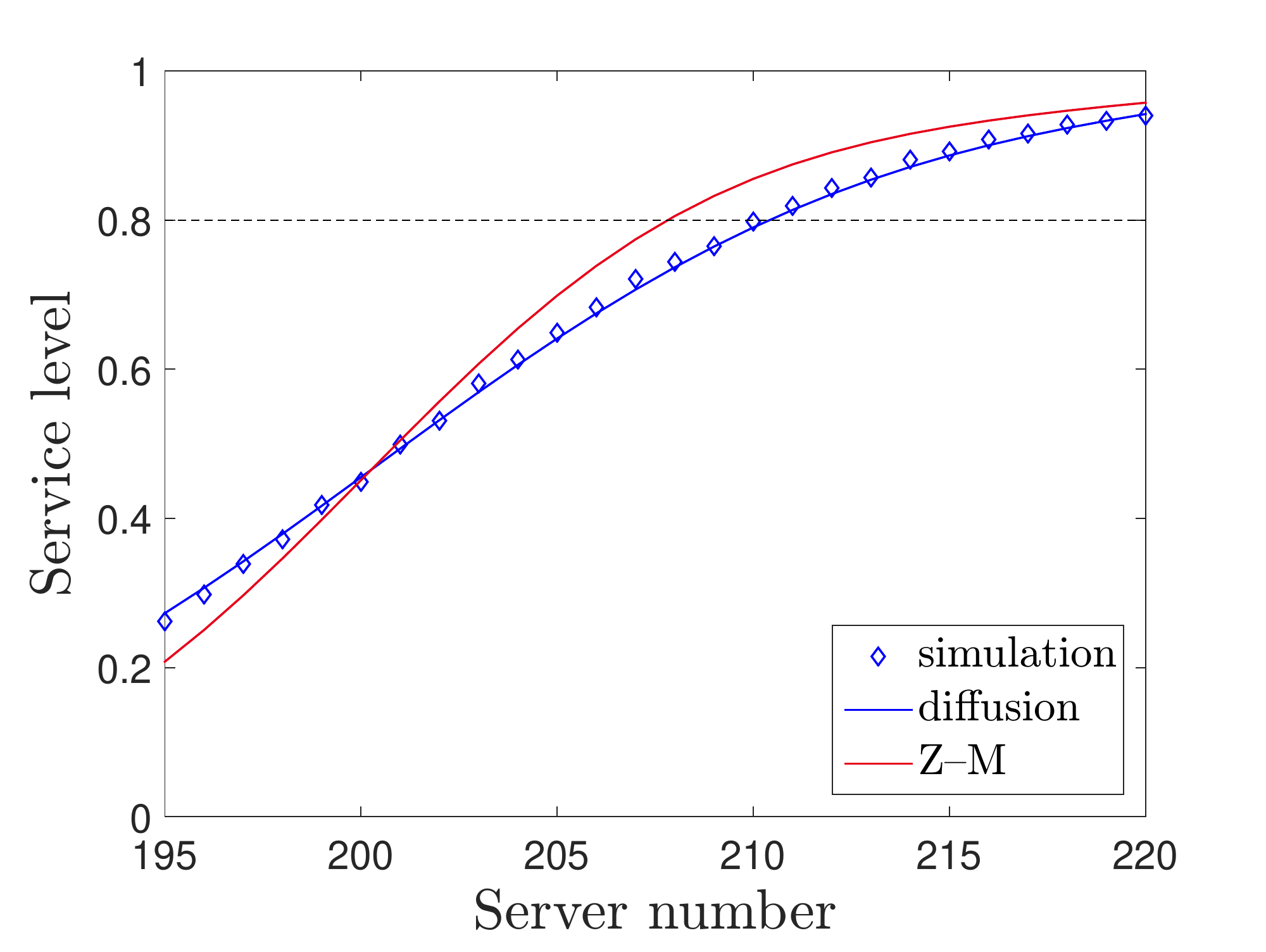}
\caption{$c_{s}^{2} = 3.0$}
\label{fig:service-level-cs3}
\end{subfigure}
\begin{subfigure}[b]{0.49\textwidth}
\includegraphics[width=3.1in]{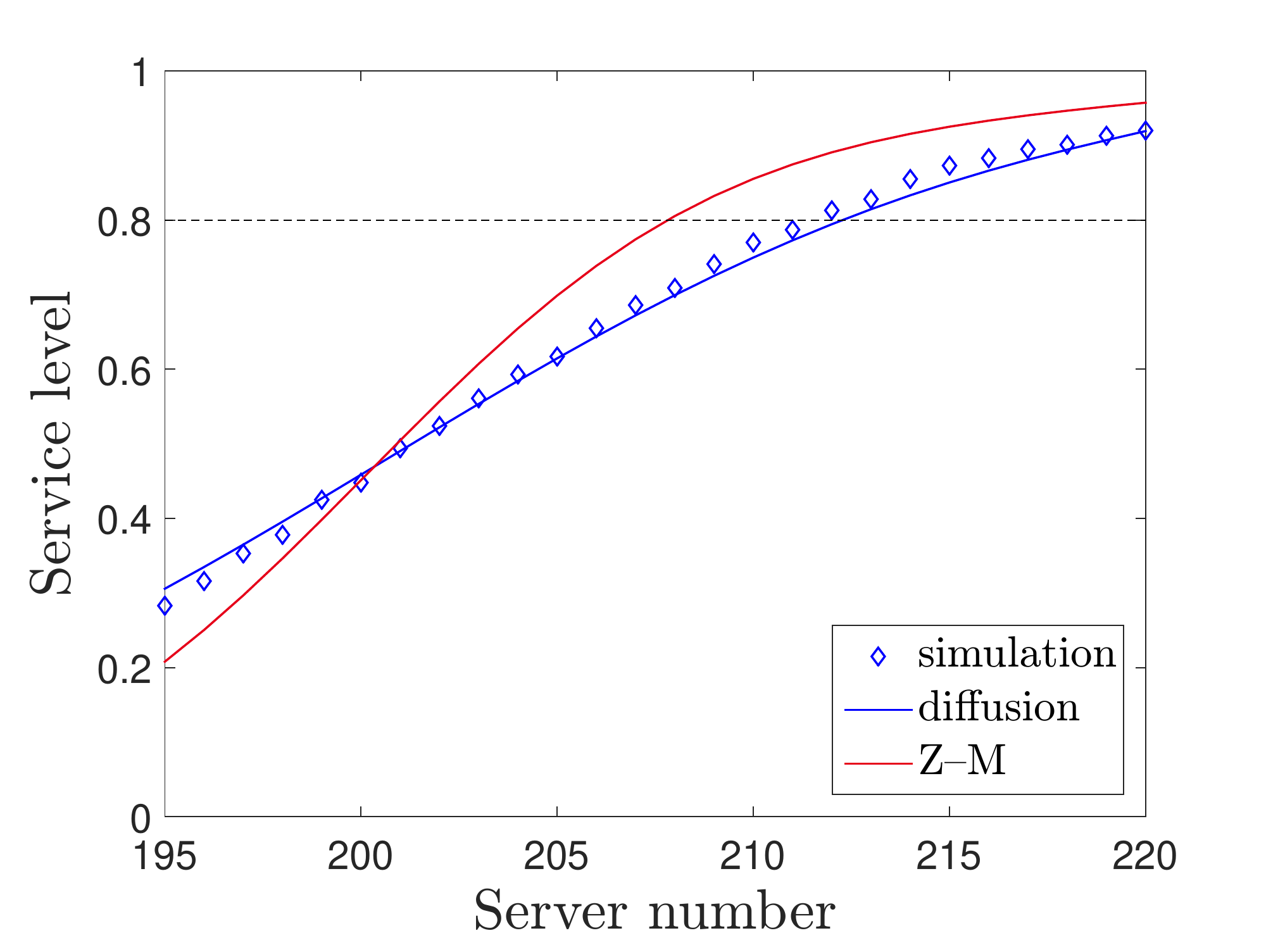}
\caption{$c_{s}^{2} = 5.0$}
\label{fig:service-level-cs5}
\end{subfigure}
\caption{The service level (the percentage of customers who receive service
within 120 seconds) in the $ \mbox{M}/\mbox{LN}/n+\mbox{H}_{2} $ queue
with $\lambda = 1.0 $, $ \mu = 1/230 $, and the hyperexponential
patience time distribution illustrated in Figure~\ref{fig:H2-patience};
simulation results are compared with the estimates by
\eqref{eq:percentage} and the estimates by
\citet{ZeltynMandelbaum05}.}
\label{fig:service-level}
\end{figure}

Let the customer arrival rate be fixed at $ \lambda = 1.0 $ per second. We would
like to determine the minimum number of servers that is required for \emph{at
least 80\% of customers to receive service within 120 seconds.} The given delay
in this example is around one half of the mean service time, which is a
reasonable requirement because customer waiting times should be comparable to
service times in order for the queue to be in the ED regime. For any fixed
number of servers, the service level within the given delay can be estimated by
\eqref{eq:percentage}.  We compare the service level estimates with simulation
results in Figure~\ref{fig:service-level}, and find good agreement in both
scenarios. To reach the service level objective, the estimates from the
diffusion model recommend $ 211 $ servers for $ c_{s}^{2} = 3.0$ and $ 213 $
servers for $ c_{s}^{2} = 5.0$, while the simulation results recommend $ 211 $
and $ 212 $ servers, respectively. It can be seen that the estimation error is
more apparent when $ c_{s}^{2} = 5.0 $, which is consistent with the fact that
this scenario has a larger SVPR.

An existing method is also considered for the purpose of comparison.
\citet{ZeltynMandelbaum05} modeled a call center as an $
\mbox{M}/\mbox{M}/n+\mbox{GI} $ queue and studied performance approximations for
this model. Their results are based on the analysis of a Markov chain studied by
\citet{BaccelliHebuterne81}, and some results are used for call center staffing
in \citet{MandelbaumZeltyn09}. The approximate distribution of the steady-state
virtual waiting time can be found in (6.12) in their paper. This approximation
turns out to be identical to our Gaussian approximation if we take $ c_{a}^{2} =
c_{s}^{2} = 1.0$ in \eqref{eq:W-tilde-tail}, so we may still use
\eqref{eq:percentage} to estimate the service level by their approximation. The
resulting service level estimates are marked with ``Z--M'' in
Figure~\ref{fig:service-level}.

\begin{figure}[t]
\centering
\begin{subfigure}[b]{0.49\textwidth}
\includegraphics[width=3.1in]{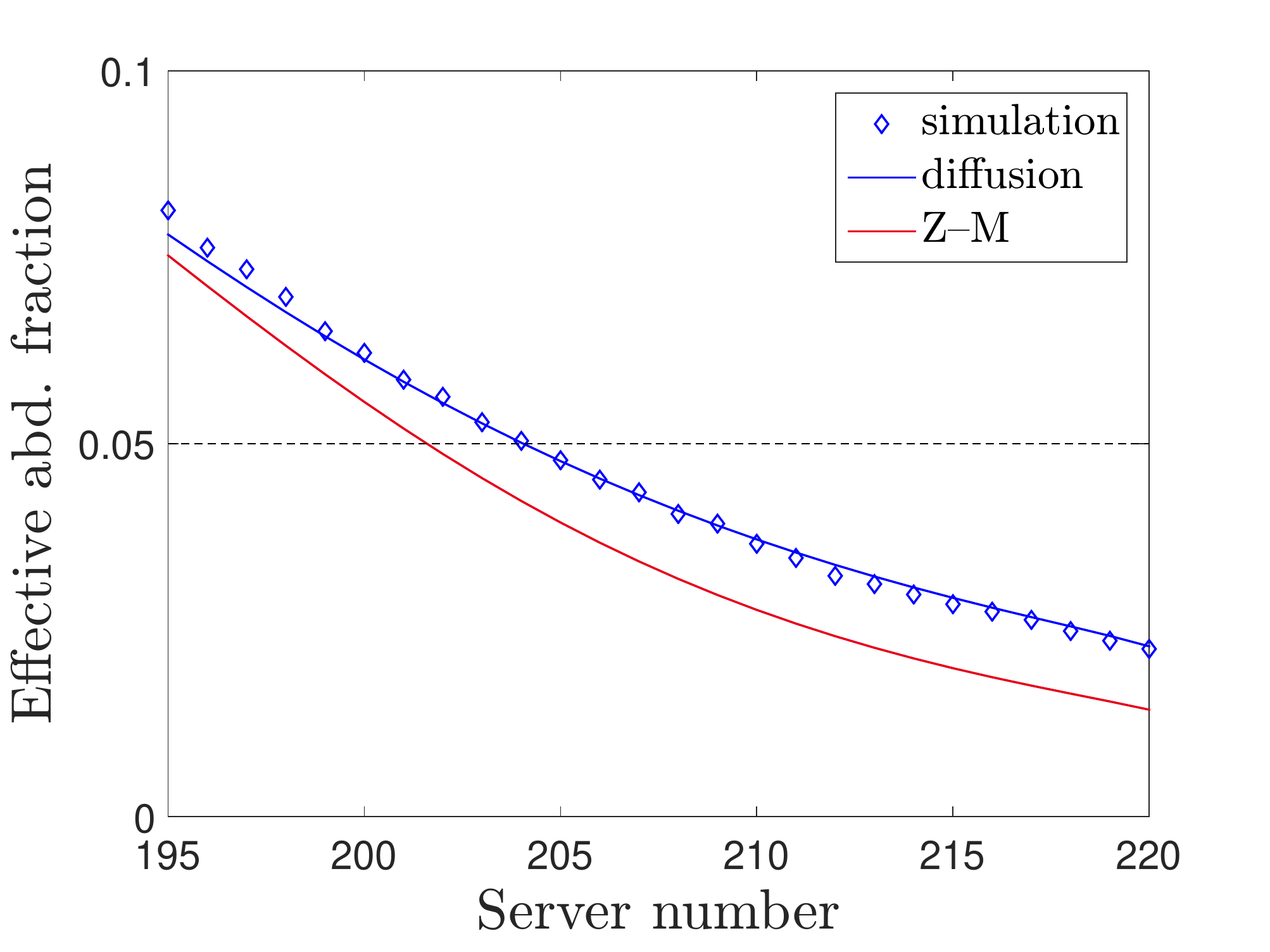}
\caption{$c_{s}^{2} = 3.0$}
\label{fig:abd-level-cs3}
\end{subfigure}
\begin{subfigure}[b]{0.49\textwidth}
\includegraphics[width=3.1in]{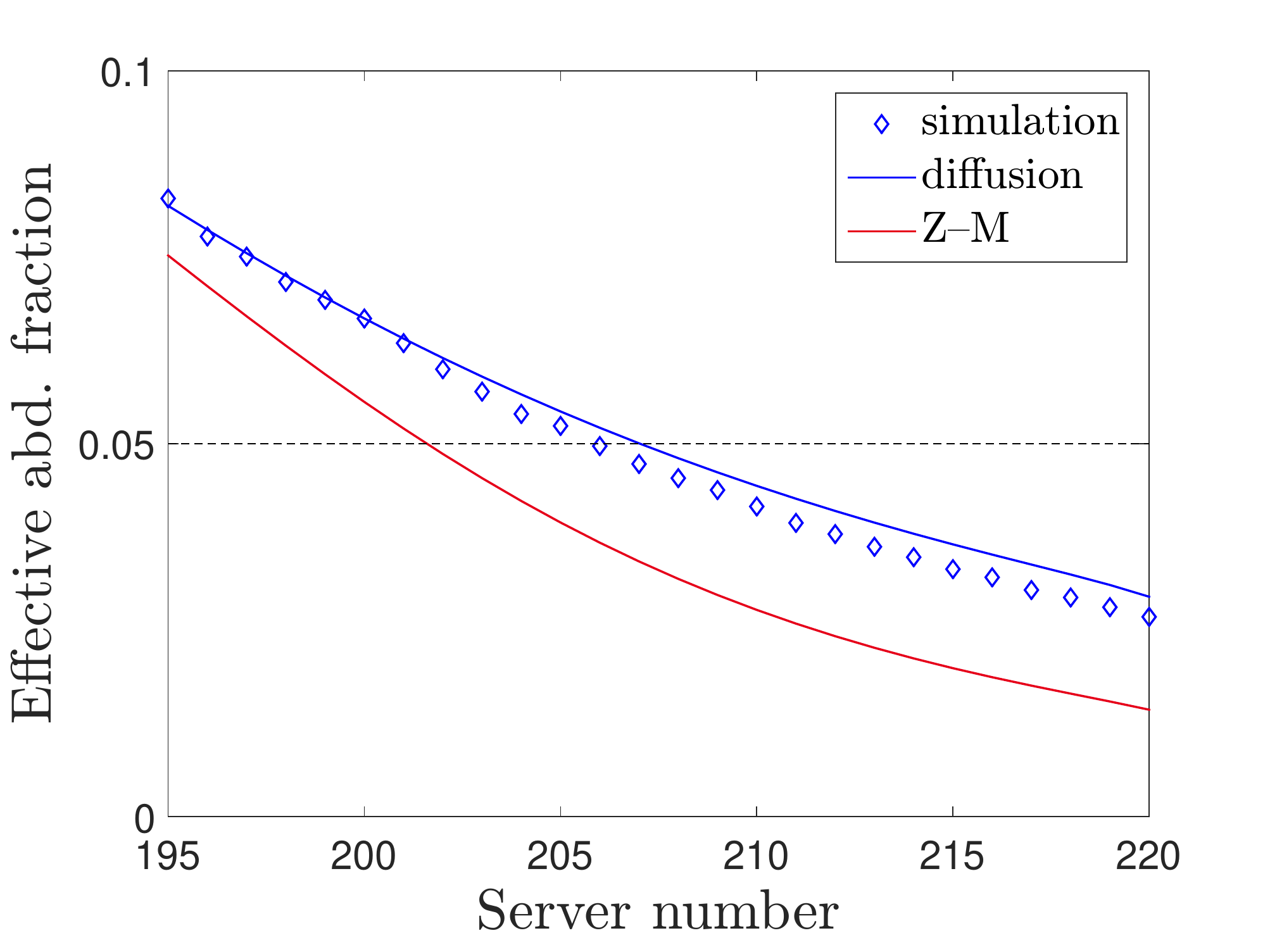}
\caption{$c_{s}^{2} = 5.0$}
\label{fig:abd-level-cs5}
\end{subfigure}
\caption{The effective abandonment fraction (the fraction of abandoning
	customers out of those whose waiting times exceed 60 seconds) in the $
	\mbox{M}/\mbox{LN}/n+\mbox{H}_{2} $ queue with $ \lambda = 1.0 $, $ \mu =
	1/230$, and the hyperexponential patience time distribution illustrated in
	Figure~\ref{fig:H2-patience}; simulation results are compared with the estimates by \eqref{eq:fraction-abandonment} and the estimates by
	\citet{ZeltynMandelbaum05}.}
\label{fig:abd-level}
\end{figure}

The approximate results by \citet{ZeltynMandelbaum05} are based on the
assumption that both interarrival times and service times are exponentially
distributed. Their estimates may not be accurate when these distributions are
non-exponential. We find a wide discrepancy in Figure~\ref{fig:service-level}
between the simulation results and their service level estimates. Because their
approximation assumes $ c_{s}^{2} = 1.0 $, the estimates suggest 208 servers in
both scenarios. If this recommendation is adopted, the actual service level will
be 74.4\% for $ c_{s}^{2} = 3.0 $ and 70.9\% for $ c_{s}^{2} = 5.0 $; in neither
case can the service level objective be fulfilled. When the variability in
arrival and service times is relatively large, one tends to underestimate the
staffing level using the $ \mbox{M}/\mbox{M}/n+\mbox{GI} $ model, in which case
the required service level can hardly be achieved.

The fraction of abandoning customers is also a common concern in call center
staffing. This time we would like to determine the minimum number of servers in
the $ \mbox{M}/\mbox{LN}/n+\mbox{H}_{2} $ queue, such that \emph{out of those
whose waiting times exceed 60 seconds, the fraction of abandoning customers is
less than 5\%.} Excluding customers abandoning the system within a short delay,
this effective abandonment fraction is used more widely as a measure of customer
satisfaction than the total percentage of abandoning customers.
Figure~\ref{fig:abd-level} compares the respective estimates by
\eqref{eq:fraction-abandonment} and by \citet{ZeltynMandelbaum05} with
simulation results. As in the previous example, the estimates by
\eqref{eq:fraction-abandonment} agree well with the simulation results, whereas
the results by \citet{ZeltynMandelbaum05} cannot capture the actual fractions.
To achieve the above target, the simulation results recommend 205 servers for $
c_{s}^{2} = 3.0 $ and 206 servers for $ c_{s}^{2} = 5.0 $, and the estimates by
\eqref{eq:fraction-abandonment} suggest 205 and 207 servers, respectively. In
contrast, the results by \citet{ZeltynMandelbaum05} yield 202 servers in both
scenarios. If one follows this recommendation, the effective abandonment
fraction will be 5.55\% for $ c_{s}^{2} = 3.0 $ and 6.15\% for $ c_{s}^{2} = 5.0
$, failing to reach the performance objective in both cases.

\section{Diffusion limit of virtual waiting time}
\label{sec:Proof-virtual}

This section is dedicated to the proof of
Theorem~\ref{theorem:many-server-virtual}. A sequence of perturbed systems is
introduced in Section~\ref{sec:perturbed}. In Section~\ref{sec:equivalence}, we
first show that the perturbed systems are asymptotically equivalent to the
original queues, then prove the diffusion limit for virtual waiting time
processes in the perturbed systems, and finally finish the proof of the theorem
by using the asymptotic equivalence. The proof procedure in
Section~\ref{sec:equivalence} partially follows that of Theorem~1 in
\citet{HuangETAL14}.

\subsection{Perturbed systems}
\label{sec:perturbed}

Consider the virtual waiting time process in the $n$th queue. Let $w_{n,k}$ be
the \emph{offered waiting time} of the $k$th customer arriving after time zero,
which is the amount of time the $k$th customer would wait until getting
into service, provided that his patience is infinite. If we use $\zeta_{n,k}$ to
denote the patience time of the $k$th customer, the number of customers who
arrived during $(0,t]$ but will eventually abandon the system can be counted by
\[
L_{n}(t)=\sum_{k=1}^{E_{n}(t)}1_{\{\zeta_{n,k}\leq w_{n,k}\}}.
\]
Suppose that all customers arriving after time $t$ are rejected immediately.
Then, the virtual waiting time at $t$ turns out to be the amount of time from
$t$ until an idle server appears, i.e.,
\begin{equation}
W_{n}(t)=\inf\{u\geq0:X_{n}(0)+E_{n}(t)-L_{n}(t)-D_{n}(t+u)<n\},
\label{eq:Wn}
\end{equation}
where $D_{n}(t)$ is the number of service completions during $(0,t]$. Let
$a_{n,k}$ be the arrival time of the $k$th customer. Because no two customers
arrive at the same time, it follows from Lemma~3.2 in \citet{DaiHe10} that
\[
w_{n,k}=W_{n}(a_{n,k}-).
\]

All servers are almost always busy in the ED regime. For $j=1,\ldots,n$, because
$\{\xi_{j,k}:k\in\mathbb{N}\}$ is the sequence of service times to be finished
by the $j$th server, the service completion process from this server is
identical to the renewal process $N_{j}$ until it begins to idle. Therefore, the
departure process $D_{n}$ is identical to the superposition of
$N_{1},\ldots,N_{n}$ until the first idle server appears. Let
\[
\tau_{n}=\inf\{t\geq0:W_{n}(t)=0\}
\]
be the time that the first idle server appears. Then, $ \tau_{n} > 0 $ because
all servers are busy at time zero. The departure process satisfies
\[
D_{n}(t)=B_{n}(t)\quad\mbox{for }0\leq t\leq\tau_{n},
\]
where $B_{n}(t)$ is given by \eqref{eq:B}. As the superposition of $n$ iid
stationary renewal processes, $B_{n}$ is more analytically tractable than
$D_{n}$. The equivalence between these two processes up to time $\tau_{n}$
allows us to explore a perturbed system that has simplified dynamics. This
perturbed system is asymptotically equivalent to the original queue as $n$ goes
large.

Because $W_{n}(t)>0$ for $0\leq t<\tau_{n}$, it follows from \eqref{eq:Wn} that with probability one,
\[
X_{n}(0)+E_{n}(t)-L_{n}(t)-D_{n}(t+W_{n}(t))=n-1\quad\text{for }0\leq
t\leq\tau_{n}.
\]
Because $D_{n}(t+u)=B_{n}(t+u)$ for $0\leq t\leq\tau_{n}$ and $0\leq u\leq
W_{n}(t)$, we further have
\[
X_{n}(0)+E_{n}(t)-L_{n}(t)-B_{n}(t+W_{n}(t))=n-1\quad\text{for }0\leq
t\leq\tau_{n}.
\]
Let us introduce a new process $V_{n}$ by using a slightly modified dynamical
equation
\begin{equation}
X_{n}(0)+E_{n}(t)-R_{n}(t)-B_{n}(t+V_{n}(t))=n-1\quad\text{for }t\geq0,
\label{eq:perturbed}%
\end{equation}
where
\[
R_{n}(t)=\sum_{k=1}^{E_{n}(t)}1_{\{\zeta_{n,k}\leq V_{n}(a_{n,k}-)\}}.
\]
By a standard sample path argument, one can show that \eqref{eq:perturbed} has a unique solution $V_{n}$. Clearly,
\begin{equation}
W_{n}(t)=V_{n}(t)\quad\text{for }0\leq t\leq\tau_{n}
\label{eq:VW}
\end{equation}
on each sample path, and $\tau_{n}$ can thus be defined alternatively by%
\begin{equation}
\tau_{n}=\inf\{t\geq0:V_{n}(t)=0\}. \label{eq:tau-n}%
\end{equation}
Taking $t=0$ in \eqref{eq:perturbed}, we obtain $B_{n}(V_{n}(0))=X_{n}
(0)-(n-1)$, which allows us to write \eqref{eq:perturbed} into
\begin{equation}
B_{n}(t+V_{n}(t))=B_{n}(V_{n}(0))+E_{n}(t)-R_{n}(t).
\label{eq:BV}
\end{equation}
We refer to $V_{n}$ as the virtual waiting time process in the $ n $th
\emph{perturbed system}. The dynamical equation \eqref{eq:BV} is identical to
equation (9) in \citet{HuangETAL14}, while the latter is derived for their
original system.

The perturbed system can be envisioned as a queue where no server is allowed to
idle. If a server finds no waiting customers upon a service completion, she
begins to serve a customer who has not arrived yet. In the perturbed system, all
servers are always busy and the departure process from each server is a
stationary renewal process.

\subsection{Asymptotic equivalence}
\label{sec:equivalence}

We first prove a fluid limit for the virtual waiting time processes in the
perturbed systems. This limit allows us to establish the asymptotic equivalence
between the original queues and the perturbed systems, which implies that these
two sequences of systems have an identical diffusion limit.

Lemma~\ref{lemma:map} is a modified version of Theorem~4.1 in
\citet{PangETAL07}, stating the continuity of a map defined by an
integral equation. This map will be used extensively in the subsequent proofs.
\begin{lemma}
\label{lemma:map}
For any $f\in\mathbb{D}$, let $x$ be a function in $\mathbb{D}$ such that
\begin{equation}
x(t)= f(t)- \int_{0}^{t} g(x(u))\,\mathrm{d}u,
\label{eq:map}
\end{equation}
where $ g:\mathbb{R}\rightarrow\mathbb{R} $ is Lipschitz continuous with $
g(0)=0 $. Then for each $ f\in\mathbb{D} $, there is a unique $x\in\mathbb{D}$
such that \eqref{eq:map} holds. Let $\psi:\mathbb{D}\rightarrow\mathbb{D}$ be
the function that maps $f$ to $x$. Then, $\psi$ is a continuous map when
$\mathbb{D}$ (as both the domain and the range) is endowed with the $J_{1}$
topology.
\end{lemma}

Write the dynamical equation \eqref{eq:BV} into the fluid-scaled form
\begin{equation}
\bar{B}_{n}(t+\bar{V}_{n}(t))=\bar{B}_{n}(\bar{V}_{n}(0))+\bar{E}_{n}%
(t)-\bar{R}_{n}(t),
\label{eq:fluid-BV}
\end{equation}
where
\begin{equation}
\bar{E}_{n}(t)=\frac{1}{n\gamma_{n}}E_{n}(\gamma_{n}t),
\quad \bar{B}_{n}(t)=\frac{1}{n\gamma_{n}}B_{n}(\gamma_{n}t), \quad\bar{R}_{n}(t)=\frac{1}{n\gamma_{n}
}R_{n}(\gamma_{n}t),
\quad \bar{V}_{n}(t)=\frac{1}{\gamma_{n}}V_{n}(\gamma_{n}t).
\label{eq:bar-VR}
\end{equation}
By \eqref{eq:E}, the fluid-scaled arrival process satisfies
\begin{equation}
\bar{E}_{n}\Rightarrow\rho\mu e\quad\mbox{as }n\rightarrow\infty.
\label{eq:fluid-E}
\end{equation}
As a scaled version of the superposition of renewal processes, $ \bar{B}_{n}
$ satisfies a functional strong law of large numbers (FSLLN), as is stated
in the proposition below. We leave the proof to the appendix.

\begin{proposition}
\label{prop:FSLLN}
Under the conditions of Theorem~\ref{theorem:FCLT},
\[
\bar{B}_{n}\rightarrow\mu e\quad\mbox{almost surely as }n\rightarrow\infty\mbox{.}
\]
\end{proposition}

The following result states that in these perturbed systems, both the
fluid-scaled virtual waiting time processes and their total variations are
stochastically bounded.

\begin{lemma}
\label{lemma:TV}
Under the conditions of Theorem~\ref{theorem:many-server-virtual},
\begin{equation}
\lim_{a\rightarrow\infty}\limsup_{n\rightarrow\infty}\mathbb{P}\Big[\sup
_{0\leq t\leq T}(t+\bar{V}_{n}(t))>a\Big]=0
\label{eq:V-bar-bounded}
\end{equation}
and
\begin{equation}
\lim_{a\rightarrow\infty}\limsup_{n\rightarrow\infty}\mathbb{P}\Big[\int%
_{0}^{T}|\mathrm{d}\bar{V}_{n}(t)|>a\Big]=0 \quad\mbox{for }T>0.
\label{eq:V-total}
\end{equation}
\end{lemma}

The next lemma
establishes an asymptotic relationship in the fluid scaling, allowing us to
approximate the abandonment process by an integral of the virtual waiting time
process.
\begin{lemma}
\label{lemma:fluid-R}
Under the conditions of Theorem~\ref{theorem:many-server-virtual},%
\[
\sup_{0\leq t\leq T}\Big\vert\bar{R}_{n}(t)-\rho\mu\int_{0}^{t}H(\bar{V}%
_{n}(u))\,\mathrm{d}u\Big\vert\Rightarrow0\quad\text{for }T>0.
\]
\end{lemma}

Now we can prove the fluid limit for the virtual waiting
time processes in the perturbed systems.
\begin{proposition}
\label{prop:fluid-V}
Under the conditions of Theorem~\ref{theorem:many-server-virtual},
\[
\bar{V}_{n}\Rightarrow \bar{w}\chi\quad\text{as }n\rightarrow\infty.
\]
\end{proposition}

\begin{proof}
By \eqref{eq:fluid-BV}, we obtain
\[
\bar{V}_{n}(t)=\bar{V}_{n}(0)+\bar{I}_{n}(t)+(\rho-1)t-\rho\int_{0}^{t}
H(\bar{V}_{n}(u))\,\mathrm{d}u,
\]
where
\begin{align*}
\bar{I}_{n}(t) &
=\frac{1}{\mu}\big(\bar{B}_{n}(\bar{V}_{n}(0))-\mu\bar{V}_{n}(0)\big)
-\frac{1}{\mu}\big(\bar{B}_{n}(t+\bar{V}_{n}(t))-\mu(t+\bar{V}_{n}(t))\big)
+\frac{1}{\mu}(\bar{E}_{n}(t)-\rho\mu t)\\
& \quad -\frac{1}{\mu}\Big(\bar{R}_{n}(t)-\rho\mu\int_{0}^{t}H(\bar{V}_{n}(u))
\,\mathrm{d}u\Big).
\end{align*}
Then, $\bar{V}_{n}=\psi(\bar{V}_{n}(0)\chi+\bar{I}_{n}+(\rho-1)e)$, where $ \psi
$ is the map defined by \eqref{eq:map} with $g(u)=\rho H(u)$ for $
u\in\mathbb{R} $. By \eqref{eq:bounded-density}, $ H $ is Lipschitz continuous
with $ H(0)=0 $, so $ \psi $ is a continuous map. Because
$\bar{w}\chi=\psi(\bar{w}\chi+(\rho-1)e)$, Lemma~\ref{lemma:map} and the
continuous mapping theorem (see, e.g., Theorem~5.2 in \citet{ChenYao01}) will
lead to the assertion once we prove $\bar{V}_{n}(0)\Rightarrow \bar{w}$ and
$\bar{I} _{n}\Rightarrow0$ as $n\rightarrow\infty$.

The convergence of $\bar{V}_{n}(0)$ follows from \eqref{eq:initial} and
\eqref{eq:VW}, and the convergence of $\bar{I}_{n}$ follows from
\eqref{eq:fluid-E}, Proposition~\ref{prop:FSLLN}, and
Lemmas~\ref{lemma:map}--\ref{lemma:fluid-R}.
\end{proof}

The fluid limit allows us to establish the asymptotic equivalence between the perturbed systems and the original queues. Put $ \bar{\tau}_{n} = \tau_{n}/\gamma_{n} $, which is the instant when the first idle server appears in
the $ n $th time-scaled perturbed system. The next proposition states that $ \bar{\tau}_{n}  $ goes to infinity in probability as $ n $ increases.

\begin{proposition}
\label{prop:tau_n}
Under the conditions of Theorem~\ref{theorem:many-server-virtual},
\[
\lim_{n\rightarrow\infty}\mathbb{P}[\bar{\tau}_{n}\leq T] =0\quad
\mbox{for all }T>0.
\]
\end{proposition}

\begin{proof}
By \eqref{eq:tau-n} and \eqref{eq:bar-VR}, $\bar{\tau
}_{n}=\inf\{t\geq0:\bar{V}_{n}(t) = 0\}$, which yields%
\[
\mathbb{P}[\bar{\tau}_{n}\leq T]=\mathbb{P}\Big[\inf_{0\leq t\leq T}\bar
{V}_{n}(t) = 0\Big].
\]
Then, the assertion follows from Proposition~\ref{prop:fluid-V}.
\end{proof}
\begin{remark}
By \eqref{eq:w-bar}, \eqref{eq:VW}, and Propositions~\ref{prop:fluid-V} and~\ref{prop:tau_n}, we obtain
\begin{equation}
\bar{W}_{n} \Rightarrow \bar{w}\chi\quad\mbox{as }n\rightarrow\infty,
\label{eq:Wn-bar-limit}
\end{equation}
which is the fluid limit for the virtual waiting time processes in the original
queues.
\end{remark}

The diffusion-scaled virtual waiting time process in the $ n $th perturbed
system is defined by
\begin{equation}
\tilde{V}_{n}(t)=\sqrt{\frac{n}{\gamma_{n}}}(V_{n}(\gamma_{n}t)-
\gamma_{n}\bar{w}).
\label{eq:V-tilde}
\end{equation}
Then,
\begin{equation}
	\tilde{W}_{n}(t)=\tilde{V}_{n}(t)\quad\text{for }0\leq t\leq\bar{\tau}_{n}.
	\label{eq:tilde-WV}
\end{equation}
Proposition~\ref{prop:tau_n} implies that $\tilde{W}_{n}$ and $\tilde{V}_{n}$
are asymptotically equal over any finite time interval. This allows us to
consider the diffusion limit of $\tilde{V}_{n}$ in order to obtain that of $
\tilde{W}_{n} $.

For notational convenience, write $ \bar{a}_{n,k} = a_{n,k}/\gamma_{n} $ and $ \bar{\zeta}_{n,k}/\gamma_{n} $, which are the scaled arrival and patience times of the $ k $th customer. Put
\[
\tilde{R}_{n}(t)=\frac{1}{\sqrt{n\gamma_{n}}}\sum_{k=1}^{E_{n}(\gamma_{n}
t)}\big(1_{\{\zeta_{n,k}\leq
V_{n}(a_{n,k}-)\}}-\Theta_{n}(V_{n}(a_{n,k}-))\big).
\]
By \eqref{eq:B_tilde}, \eqref{eq:bar-VR}, and \eqref{eq:V-tilde}, we can write \eqref{eq:BV} into the diffusion-scaled
form
\begin{equation}
\tilde{V}_{n}(t)=\tilde{V}_{n}(0)+\tilde{M}_{n}(t)-\tilde{\Upsilon}
_{n}(t)-\rho f_{H}(\bar{w})\int_{0}^{t}\tilde{V}_{n}(u)\,\mathrm{d}u,
\label{eq:Vn-diffusion}
\end{equation}
where%
\begin{equation}
\tilde{M}_{n}(t)=\frac{1}{\mu}\big(\tilde{B}_{n}(\bar{V}_{n}(0))-\tilde{B}_{n}(t+\bar{V}_{n}(t))-\tilde{R}_{n}(t)\big)+\frac{1}{\rho\mu}\tilde{E}_{n}(t)
\label{eq:M-tilde}
\end{equation}
and
\begin{equation}
\tilde{\Upsilon}_{n}(t)=\frac{1}{\mu\sqrt{n\gamma_{n}}}\sum_{k=1}
^{E_{n}(\gamma_{n}t)}\left(  H(\bar{V}_{n}(\bar{a}_{n,k}-))-H(\bar{w})\right)
-\rho f_{H}(\bar{w})\int_{0}^{t}\tilde{V}_{n}(u)\,\mathrm{d}u.
\label{eq:Upsilon}
\end{equation}
Equations \eqref{eq:Vn-diffusion}--\eqref{eq:Upsilon} follow equations
(16)--(17) in \citet{HuangETAL14}.

To obtain the limit process of $ \tilde{V}_{n} $, we should first obtain the limit processes of $ \tilde{M}_{n} $ and $ \tilde{\Upsilon}_{n} $. The convergence
of $ \tilde{M}_{n} $ can be deduced by the next lemma, which states the
joint convergence of the diffusion-scaled arrival, service completion, and
abandonment processes in the perturbed systems.
\begin{lemma}
\label{lemma:EBR}
Under the conditions of Theorem~\ref{theorem:many-server-virtual},
\[
(\tilde{E}_{n},\tilde{B}_{n},\tilde{R}_{n})\Rightarrow(\hat{E},\hat
{B},\hat{R})\quad\text{as }n\rightarrow\infty,
\]
where $\hat{E}$, $\hat{B}$, and $\hat{R}$ are independent
driftless Brownian motions with variances $\rho\mu c_{a}^{2}$, $\mu
c_{s}^{2}%
$, and $(\rho-1)\mu/\rho$, respectively.
\end{lemma}

The process $ \tilde{\Upsilon}_{n} $ turns out to be an error term converging to zero.

\begin{lemma}
\label{lemma:HV}
Under the conditions of Theorem~\ref{theorem:many-server-virtual},
\[
\tilde{\Upsilon}_{n}\Rightarrow0\quad\text{as }n\rightarrow\infty.
\]
\end{lemma}

Using these convergence results, we can obtain the diffusion limit for the perturbed systems.
\begin{proposition}
\label{prop:virtual-limit-perturbed}
Under the conditions of Theorem~\ref{theorem:many-server-virtual},
\[
\tilde{V}_{n}\Rightarrow\hat{W}\quad\mbox{as }n\rightarrow\infty,
\]
where $ \hat{W} $ is the OU process given by \eqref{eq:OU-limit}.
\end{proposition}

\begin{proof}
By \eqref{eq:initial} and \eqref{eq:tilde-WV}, $ \tilde{V}_{n}(0) \Rightarrow
\hat{W}(0) $ as $ n \rightarrow \infty $. Using Proposition~\ref{prop:fluid-V} and Lemma~\ref{lemma:EBR}, we obtain  $
\tilde{M}_{n} \Rightarrow \hat{M} $ as $ n \rightarrow \infty $, where $ \hat{M} $
is a driftless Brownian motion with $\hat{M}(0) = 0$ and variance $ (c_{a}^{2} + \rho c_{s}^{2} +
\rho-1)/(\rho\mu) $. Then, by Lemma~\ref{lemma:HV},
\[
\tilde{V}_{n}(0)\chi+\tilde{M}_{n}-\tilde{\Upsilon}_{n} \Rightarrow \hat{W}(0)\chi + \hat{M} \quad\mbox{as }n\rightarrow\infty.
\]
Because $ \tilde{V}_{n} = \psi(\tilde{V}_{n}(0)\chi + \tilde{M}_{n} -
\tilde{\Upsilon}_{n}) $ and $ \hat{W} = \psi(\hat{W}(0)\chi + \hat{M}) $, where
$ \psi $ is the map given by \eqref{eq:map} with $ g(u)=\rho f_{H}(\bar{w})u$
for $ u\in\mathbb{R} $, the assertion follows from Lemma~\ref{lemma:map} and
the continuous mapping theorem.
\end{proof}

Let us complete the proof of Theorem~\ref{theorem:many-server-virtual} by using the asymptotic equivalence.
\begin{proof}
[Proof of Theorem~\ref{theorem:many-server-virtual}]
By \eqref{eq:tilde-WV},
\[
\mathbb{P}\Big[\sup_{0\leq t\leq
T}|\tilde{W}_{n}(t)-\tilde{V}_{n}(t)|>0\Big]\leq\mathbb{P}
[\bar{\tau}_{n}\leq T]\quad\mbox{for }T>0.
\]
Then, Proposition~\ref{prop:tau_n} implies that $\tilde{W}_{n}-\tilde{V}_{n}
\Rightarrow 0$
as $n \rightarrow \infty$. The theorem follows from
Proposition~\ref{prop:virtual-limit-perturbed} and
the convergence-together theorem (see Theorem~5.4 in \citet{ChenYao01}).
\end{proof}

\section{Gaussian limit of queue length}
\label{sec:Proof-queue}

We prove Theorem~\ref{theorem:many-server} in this section. With a general
patience time distribution, it is difficult to prove a diffusion limit for the
queue length process in the ED regime.  \citet{HuangETAL14} proved a
one-dimensional limit process of queue length process for the $
\mbox{GI}/\mbox{M}/n+\mbox{GI} $ model with a general scaling of patience-time
distributions; they used the diffusion limit of virtual waiting time to obtain
the limit queue length process. In the present study, we consider the marginal
distribution of the queue length at a particular time, which also allows us to
infer the steady-state queue length distribution.

Let $ Q_{n}(t) = (X_{n}(t)-n)^{+}$ be the queue length in the $ n $th system at
time $ t $. Because customers are first-come first-served, those who arrived
before $ t $ must have either entered service or abandoned the system by $
t+W_{n}(t) $. In other words, customers who are waiting at $t+W_{n}(t)$ must
have arrived during $(t,t+W_{n}(t)]$ and have not abandoned the system by
$t+W_{n}(t)$. This observation yields
\[
Q_{n}(t+W_{n}(t)) = \sum_{k=E_{n}(t)+1}^{E_{n}(t+W_{n}(t))} 1_{\{\zeta_{n,k} > t+W_{n}(t)-a_{n,k}\}}.
\]
For $0\leq t<\tau_{n}$, because $Q_{n}(t+W_{n}(t))=X_{n}(t+W_{n}(t))-n$, it
follows from \eqref{eq:tilde-xn} and \eqref{eq:w-bar} that
\[
\tilde{X}_{n}(t+\bar{W}_{n}(t))=\frac{1}{\sqrt{n\gamma_{n}}}(Q_{n}(\gamma
_{n}t+W_{n}(\gamma_{n}t))-q_{n}),
\]
where $q_{n}$ is given by \eqref{eq:qn}. Then,
\[
\tilde{X}_{n}(t+\bar{W}_{n}(t))=\frac{1}{\sqrt{n\gamma_{n}}}\sum
_{k=E_{n}(\gamma_{n}t)+1}^{E_{n}(\gamma_{n}t+W_{n}(\gamma_{n}t))}
1_{\{\zeta_{n,k}>\gamma_{n}t+W_{n}(\gamma_{n}t)-a_{n,k}\}}-\sqrt{n\gamma_{n}%
}\rho\mu\int_{0}^{\bar{w}}(1-H(u))\,\mathrm{d}u.
\]
By \eqref{eq:bar-VR}, we can further decompose the scaled
queue length into
\begin{equation}
\tilde{X}_{n}(t+\bar{W}_{n}(t))=\mu\tilde{W}_{n}(t)+\tilde{G}_{n}^{\prime
}(t)+\tilde{G}_{n}^{\prime\prime}(t)+\tilde{Y}_{n}(t)+\tilde{Y}_{n}^{\prime
}(t),
\label{eq:tilde-x-t-w}
\end{equation}
where
\[
\tilde{G}_{n}'(t)   =\sqrt{n\gamma_{n}}\int_{t}^{t+\bar{w}%
}(1-H(t+\bar{w}-u))\,\mathrm{d}\bar{E}_{n}(u)-\sqrt{n\gamma_{n}}\rho\mu
\int_{t}^{t+\bar{w}}(1-H(t+\bar{w}-u))\,\mathrm{d}u,
\]

\begin{align*}
\tilde{G}_{n}''(t) &  =\frac{1}{\sqrt{n\gamma_{n}}}\sum_{k=E_{n}%
	(\gamma_{n}t)+1}^{E_{n}(\gamma_{n}t+W_{n}(\gamma_{n}t))}\big(  H(t+\bar
{W}_{n}(t)-\bar{a}_{n,k})-1_{\{\bar{\zeta}_{n,k}\leq t+\bar{W}_{n}(t)-\bar
	{a}_{n,k}\}}\big)  ,\\
\tilde{Y}_{n}(t) &  =\tilde{E}_{n}(t+\bar{W}_{n}(t))-\tilde{E}_{n}(t+\bar
{w})-\sqrt{n\gamma_{n}}\int_{t+\bar{w}}^{t+\bar{W}_{n}(t)}H(t+\bar{W}%
_{n}(t)-u)\,\mathrm{d}\bar{E}_{n}(u)\\
&  \quad+(\rho-1)\mu\tilde{W}_{n}(t)-\sqrt{n\gamma_{n}}\rho\mu\int%
_{0}^{\bar{w}}(H(\bar{W}_{n}(t)-u)-H(\bar{w}-u))\,\mathrm{d}u,\\
\tilde{Y}_{n}^{\prime}(t) &  =\sqrt{n\gamma_{n}}\int_{t}^{t+\bar{w}}%
(H(t+\bar{w}-u)-H(t+\bar{W}_{n}(t)-u))\,\mathrm{d}\bar{E}_{n}(u)\\
&  \quad-\sqrt{n\gamma_{n}}\rho\mu\int_{t}^{t+\bar{w}}(H(t+\bar{w}%
-u)-H(t+\bar{W}_{n}(t)-u))\,\mathrm{d}u.
\end{align*}

To obtain the limit distribution of $ \tilde{X}_{n}(t+\bar{W}_{n}(t)) $, we need
to analyze each term on the right side of \eqref{eq:tilde-x-t-w}. The joint
convergence of the first three terms is given by Lemma~\ref{lemma:WGG}.

\begin{lemma}
\label{lemma:WGG}
Under the conditions of Theorem~\ref{theorem:many-server-virtual}, for any fixed $ t\geq 0 $,
\[
(\tilde{W}_{n}(t),\tilde{G}^{\prime}_{n}(t),\tilde{G}''_{n}(t))
\Rightarrow (\hat{W}(t),\hat{G}^{\prime}(t),\hat{G}''(t))\quad \mbox{as
}n\rightarrow\infty,
\]
where $ \hat{W} $ is the OU process given by \eqref{eq:OU-limit},
\[
\hat{G}'(t) = \int_{t}^{t+\bar{w}} (1-H(t+\bar{w}-u))\,\mathrm{d}\hat{E}(u),
\]
and
\[
\hat{G}^{\prime\prime}(t)=\int_{0}^{\bar{w}}\sqrt{\rho\mu H(\bar
{w}-u)(1-H(\bar{w}-u))}\,\mathrm{d}\hat{S}(u).
\]
Here, $ \hat{S} $ is a standard Brownian motion that is independent of $ \hat{E} $, $ \hat{B} $, and $ \{\hat{W}(u):0\leq u \leq t\} $.
\end{lemma}

The other two random variables $ \tilde{Y}_{n}(t) $ and $ \tilde{Y}'_{n}(t) $ are error terms converging to zero.

\begin{lemma}
\label{lemma:YY}
Under the conditions of Theorem~\ref{theorem:many-server-virtual}, for any fixed $t \geq 0$,
\[
(\tilde{Y}_{n}(t),\tilde{Y}_{n}^{\prime}(t))\Rightarrow(0,0)\quad\mbox{as
}n\rightarrow\infty.
\]
\end{lemma}

Using the previous two lemmas, we are ready to prove
Theorem~\ref{theorem:many-server}.
\begin{proof}
	[Proof of Theorem~\ref{theorem:many-server}]
Put $ \tilde{G}_{n}(t) = \tilde{G}_{n}^{\prime
}(t)+\tilde{G}_{n}^{\prime\prime}(t)+\tilde{Y}_{n}(t)+\tilde{Y}_{n}^{\prime
}(t)$. By Lemmas~\ref{lemma:WGG} and~\ref{lemma:YY},
\[
(\tilde{W}_{n}(t),\tilde{G}_{n}(t)) \Rightarrow (\hat{W}(t), \hat{G}(t))
\quad\mbox{as }n\rightarrow \infty,
\]
from which the convergence of $ \tilde{X}_{n}(t+\bar{W}_{n}(t)) $ follows.
\end{proof}

\section{Conclusion and future work}
\label{sec:conclusion}

A diffusion model was proposed for $ \mbox{GI}/\mbox{GI}/n+\mbox{GI} $ queues in
the ED regime. We adopted a space-time scaling approach, in which the number of
servers and the mean patience time are used as the respective scaling factors in
space and time, to obtaining a one-dimensional diffusion approximation for the
virtual waiting time process. Using this diffusion model, we derived the
steady-state distributions of virtual waiting time and queue length, along with
approximate formulas for other performance measures. These approximations are
generally accurate when the mean patience time is several times longer than the
mean service time and the patience time distribution does not change rapidly
around the mean virtual waiting time.

One limitation of the diffusion model is as follows. As we discussed in
Section~\ref{sec:Limits}, the diffusion limit given by \eqref{eq:OU-limit}
depends on the normalized patience time distribution only through the hazard
rate at $ \bar{w} $. In consequence, all approximate formulas derived from the
diffusion model are dictated by the patience time hazard rate (or equivalently,
the patience time probability density) at the mean virtual waiting time. These
approximations will generally be accurate, if the patience time hazard rate does
not change much around the mean virtual waiting time (as the case in
Figure~\ref{fig:H2-patience}). In call centers, this assumption would be valid
if waiting customers do not receive real-time information about the queue.
However, if delay announcements are made in the call center, customers may
decide whether to hang up according to what they hear about the queue. As a
result, the patience time hazard rate may change rapidly after an announcement
time. In this case, performance approximations that depend on the patience time
distribution only through ``a single point'' may no longer produce satisfactory
results. We would thus need a diffusion model that could incorporate the
patience time hazard rate on a neighborhood of the mean virtual waiting time.
Such models were derived for $ \mbox{GI}/\mbox{M}/n+\mbox{GI} $ queues by
\citet{ReedTezcan12} in the critically loaded regime and by \citet{HuangETAL14}
in the ED regime. Furthermore, Huang et al. (2014) applied the results to the
analysis of systems with delay announcement. In the future, we would also extend
our diffusion model by including the entire patience time distribution. To this
end, the current asymptotic framework needs to be modified so that both the
space-time scaling and the hazard rate scaling used by \citet{ReedTezcan12} and
\citet{HuangETAL14} can be combined into the same asymptotic framework.

\appendix

\setcounter{section}{1}
\setcounter{equation}{0}

\section*{Appendix}

We prove Theorem~\ref{theorem:FCLT} is in Section~\ref{sec:Proof-FCLT} and prove all technical lemmas in Sections~\ref{sec:lemmas-Theorem-1} and~\ref{sec:lemmas-Theorem-2}.

\subsection{Proof of the FCLT}
\label{sec:Proof-FCLT}

Let $ S_{j,k} $ be the $k$th partial sum of
$\{\xi_{j,\ell}:\ell\in\mathbb{N}\}$, i.e.,
\[
S_{j,k}=\sum_{\ell=1}^{k}\xi_{j,\ell}\quad\mbox{for }j\in\mathbb{N}.
\]
We take $S_{j,0}=0$ by convention. Let us first present the proof of
Proposition~\ref{prop:FSLLN}.

\begin{proof}
	[Proof of Proposition~\ref{prop:FSLLN}]
Because $S_{j,N_{j}(t)}\leq t\leq S_{j,N_{j}(t)+1}$ for $t>0$, then
\[
\frac{\sum_{j=1}^{n}S_{j,N_{j}(\gamma_{n}t)}}{\sum_{j=1}^{n}N_{j}(\gamma_{n}t)}\leq\frac{n
\gamma_{n}t}{\sum_{j=1}^{n}N_{j}(\gamma_{n}t)}\leq\frac{\sum_{j=1}^{n}S_{j,N_{j}(\gamma_{n}t)+1}}
{\sum_{j=1}^{n}N_{j}(\gamma_{n}t)}
\]
provided that $\sum_{j=1}^{n}N_{j}(\gamma_{n}t)>0$. Note that
\[
\sum_{j=1}^{n}S_{j,N_{j}(\gamma_{n}t)+1}=\sum_{j=1}^{n}\xi_{j,1}+\sum_{j=1}^{n}\sum_{k=2}^{N_{j}
(\gamma_{n}t)+1}\xi_{j,k}.
\]
Because $\lim_{n\rightarrow\infty} N_{j}(\gamma_{n}t)=\infty$ almost surely for
$t>0$, then
\[
\lim_{n\rightarrow\infty}
\frac{\sum_{j=1}^{n}\sum_{k=2}^{N_{j}(\gamma_{n}t)+1}\xi_{j,k}}{\sum_{j=1}^{n}N_{j}(\gamma_{n}t)} = \frac{1}{\mu}
\]
almost surely by the strong law of large numbers. In addition, $\lim_{n
\rightarrow \infty} \sum_{j=1}^{n}N_{j} (\gamma_{n}t)/n =\infty$ almost surely
for $t>0$, which implies that
\[
\lim_{n\rightarrow\infty}\frac{\sum_{j=1}^{n}\xi_{j,1}}{\sum_{j=1}^{n}N_{j}(\gamma_{n}t)}=0.
\]
Therefore,
\[
\lim_{n\rightarrow\infty}
\frac{\sum_{j=1}^{n}S_{j,N_{j}(\gamma_{n}t)+1}}{\sum_{j=1}^{n}N_{j}(\gamma_{n}t)}
=\frac{1}{\mu}.
\]
Also,
\[
\lim_{n\rightarrow\infty}
\frac{\sum_{j=1}^{n}S_{j,N_{j}(\gamma_{n}t)}}{\sum_{j=1}^{n}N_{j}(\gamma_{n}t)}=\lim_{n\rightarrow\infty}\frac{\sum_{j=1}
^{n}S_{j,N_{j}(\gamma_{n}t)}}{\sum_{j=1}^{n}(N_{j}(\gamma_{n}t)-1)}\cdot\frac{\sum_{j=1}^{n}(N_{j}
(\gamma_{n}t)-1)}{\sum_{j=1}^{n}N_{j}(\gamma_{n}t)}=\frac{1}{\mu}.
\]
Then, $\lim_{n\rightarrow\infty}\bar{B}_{n}(t)=\mu t$ almost surely for
$t\geq0$. Because $\bar{B}_{n}(t)$ is nondecreasing in $t$ and $e$ is a
continuous function, the assertion follows from Theorem VI.2.15 in
\citet{JacodShiryaev2002}.
\end{proof}

\begin{lemma}
\label{lemma:L-tilde}
Let
\begin{equation}
\tilde{B}^{\prime}_{n}(t)=\frac{1}{\sqrt{n\gamma_{n}}}\sum_{j=1}^{n}\sum_{k=2}^{N_{j}(\gamma_{n}t)+1}(1-\mu
\xi_{j,k})\quad\mbox{for }t\geq0.\label{eq:B_check}
\end{equation}
Under the conditions of Theorem~\ref{theorem:FCLT},
\[
\tilde{B}^{\prime}_{n}\Rightarrow\hat{B}\quad\mbox{as }n\rightarrow\infty\mbox{.}
\]
\end{lemma}

\begin{proof}
Let $\{\xi^{\prime}_{k}:k\in\mathbb{N}\}$ be a sequence of iid random variables
following distribution $F$. Put
\[
\tilde{B}''_{n}(t)=\frac{1}{\sqrt{n\gamma_{n}}}\sum_{k=1}^{\lfloor
n\gamma_{n}t\rfloor}(1-\mu
\xi^{\prime}_{k})\quad\mbox{for }t\geq0.
\]
Because $\mu \xi^{\prime}_{k}$ has mean one and variance
$c_{s}^{2}$, by Donsker's theorem, $\tilde{B}''_{n}\Rightarrow\hat{B}''$ as
$n\rightarrow\infty$, where $\hat{B}''$ is a driftless Brownian motion with
$\hat{B}''(0)=0$ and variance $c_{s}^{2}$. By \eqref{eq:bar-VR},
\[
\tilde{B}''_{n}(\bar{B}_{n}(t))=\frac{1}{\sqrt{n\gamma_{n}}}
\sum_{k=1}^{n\gamma_{n}\bar{B}_{n}(t)}
(1-\mu\xi^{\prime}_{k})=\frac{1}{\sqrt{n\gamma_{n}}}
\sum_{k=1}^{N_{1}(\gamma_{n}t)+\cdots+N_{n}
(\gamma_{n}t)}(1-\mu \xi^{\prime}_{k}).
\]
Then, it follows from Proposition~\ref{prop:FSLLN} and the random
time-change theorem (see Theorem~5.3 in \citet{ChenYao01}) that $\tilde{B}''_{n} \circ
\bar{B}_{n}\Rightarrow\mu^{1/2}\hat{B}''$ as $n\rightarrow\infty$. Because
$\tilde{B}^{\prime}_{n}$ has the same distribution as
$\tilde{B}''_{n}\circ\bar{B}_{n}$ and $\mu^{1/2}\hat{B}''$ has the same
distribution as $\hat{B}$, the lemma follows.
\end{proof}

The next lemma is a technical result for proving the convergence of $
\tilde{B}_{n} $.
\begin{lemma}
\label{lemma:tightness}
Under the conditions of Theorem~\ref{theorem:FCLT},
for all $0\leq r\leq s\leq t$ and all $n\in\mathbb{N}$, there exists
$0<c<\infty$
such that
\[
\mathbb{E}[(\tilde{B}_{n}(s)-\tilde{B}_{n}(r))^{2}(\tilde{B}_{n}(t)-\tilde{B}_{n}(s))^{2}]\leq
c(t-r)^{2}.
\]
\end{lemma}

\begin{proof}
Let $\check{N}_{j}(u)=N_{j}(u)-\mu u$ for $u\geq0$ and $j=1,\ldots,n$. Because
$N_{j}$ is a stationary renewal process, by inequalities (7) and (8) in
\citet{Whitt85}, there exists $c_{1}<\infty$ such that
\begin{equation}
\mathbb{E}[(\check{N}_{j}(s)-\check{N}_{j}(r))^{2}]\leq c_{1}(s-r)\label{eq:inequality1}
\end{equation}
and
\begin{equation}
\mathbb{E}[(\check{N}_{j}(s)-\check{N}_{j}(r))^{2}(\check{N}_{j}(t)-\check{N}_{j}(s))^{2}]\leq
c_{1}(t-r)^{2}\label{eq:inequality2}
\end{equation}
for all $0\leq r\leq s\leq t$. The regularity condition (\ref{eq:conditionF})
is required for inequality \eqref{eq:inequality2} to hold. In addition, it
follows
from \eqref{eq:inequality1} and H\"{o}lder's inequality that
\begin{equation}
\mathbb{E}[|\check{N}_{j}(s)-\check{N}_{j}(r)||\check{N}_{j}(t)-\check{N}_{j}(s)|]\leq c_{1}(s-
r)^{1/2}(t-s)^{1/2}\leq c_{1}(t-r).\label{eq:inequality3}
\end{equation}
Because $N_{1},\ldots,N_{n}$ are iid processes,
\begin{align*}
 \mathbb{E}[(\tilde{B}_{n}(s)-\tilde{B}_{n}(r))^{2}
 (\tilde{B}_{n}(t)-\tilde{B}_{n}(s))^{2}]
 & =\frac{1}{n\gamma_{n}^{2}}
 \mathbb{E}[(\check{N}_{1}(\gamma_{n}s)-\check{N}_{1}(\gamma_{n}
r))^{2}(\check{N}_{1}(\gamma_{n}t)-\check{N}_{1}(\gamma_{n}s))^{2}]\\
& \quad+\frac{n-1}{n\gamma_{n}^{2}}
\mathbb{E}[(\check{N}_{1}(\gamma_{n}s)-\check{N}_{1}
(\gamma_{n}r))^{2}]
\mathbb{E}[(\check{N}_{1}(\gamma_{n}t)-\check{N}_{1}(\gamma_{n}s))^{2}]\\
& \quad+\frac{2(n-1)}{n\gamma_{n}^{2}}
\mathbb{E}[(\check{N}_{1}(\gamma_{n}s)-\check{N}_{1}
(\gamma_{n}r))(\check{N}_{1}(\gamma_{n}t)-\check{N}_{1}(\gamma_{n}s))]^{2}\\
 & \leq c_{1}(t-r)^{2}+c_{1}^{2}(s-r)(t-s)+2c_{1}^{2}(t-r)^{2},
\end{align*}
in which the inequality is obtained by
\eqref{eq:inequality1}--\eqref{eq:inequality3}. The lemma follows with
$c=3c_{1}^{2}+c_{1}$.
\end{proof}

\begin{proof}[Proof of Theorem~\ref{theorem:FCLT}.]

For $j\in\mathbb{N}$, let
\begin{equation}
U_{j}(t)=S_{j,N_{j}(t)+1}-t\label{eq:recess}
\end{equation}
be the residual time at $t\geq0$. Note that $
U_{j}(0)=\xi_{j,1}$.
Because $N_{1},\ldots,N_{n}$ are iid stationary renewal processes,
$U_{1}(t),\ldots,U_{n}(t)$ are iid random variables following distribution
$F_{e}$ for all $t\geq0$, each having mean
\[
m_{e}=\int_{0}^{\infty}t\,\mathrm{d}F_{e}(t)=\frac{1+c_{s}^{2}}{2\mu}
\]
and variance
\[
\sigma_{e}^{2}=\int_{0}^{\infty}t^{2}\,\mathrm{d}F_{e}(t)-m_{e}^{2}=\frac{\mu}{3}\int_{0}^{\infty}
t^{3}\,\mathrm{d}F(t)-m_{e}^{2}.
\]
Then,  $\sigma_{e}^{2}<\infty$ by \eqref{eq:m3}. Put
\[
\tilde{U}_{n}(t)=\frac{1}{\sqrt{n\gamma_{n}}}
\sum_{j=1}^{n}(U_{j}(\gamma_{n}t)-m_{e}),
\]
and we have
$ \mathbb{E}[\tilde{U}_{n}(t)^{2}]=
\sigma_{e}^{2}/\gamma_{n}\rightarrow 0$ as $n
\rightarrow\infty$.
This implies that $\tilde{U}_{n}(t)\Rightarrow0$ as $n\rightarrow\infty$ for
$t\geq0$. By Theorem 3.9 in \citet{Billingsley99},
\begin{equation}
(\tilde{U}_{n}(t_{1}),\ldots,\tilde{U}_{n}(t_{\ell}))\Rightarrow0\quad\mbox{as }n\rightarrow\infty
\label{eq:R_tilde}
\end{equation}
for any $\ell\in\mathbb{N}$ and $0\leq t_{1}<\cdots<t_{\ell}$. By
\eqref{eq:recess},
\[
U_{j}(t)=\xi_{j,1}+\sum_{k=2}^{N_{j}(t)+1}\xi_{j,k}-t=U_{j}(0)+\sum_{k=2}^{N_{j}(t)+1}(\xi_{j,k}-
\mu^{-1})+\mu^{-1}N_{j}(t)-t.
\]
Then, by \eqref{eq:B_tilde} and \eqref{eq:B_check}, we obtain
\begin{equation}
\tilde{B}_{n}(t)=-\mu\tilde{U}_{n}(0)+\mu\tilde{U}_{n}(t)+\tilde{B}^{\prime}_{n}(t).
\label{eq:B_tilde_composition}
\end{equation}
We deduced from \eqref{eq:R_tilde}, \eqref{eq:B_tilde_composition}, and
Lemma~\ref{lemma:L-tilde} that
\[
(\tilde{B}_{n}(t_{1}),\ldots,\tilde{B}_{n}(t_{\ell}))\Rightarrow(\hat{B}(t_{1}),\ldots,\hat{B}
(t_{\ell}))\quad\mbox{as }n\rightarrow\infty\mbox{.}
\]
Finally, it follows from Lemma~\ref{lemma:tightness} and Theorem~13.5 in
\citet{Billingsley99} (with condition (13.13) replaced by (13.14)) that
$\tilde{B}_{n}\Rightarrow\hat{B}$ as $n\rightarrow\infty$.
\end{proof}

\subsection{Proofs of lemmas for Theorem~\ref{theorem:many-server-virtual}}
\label{sec:lemmas-Theorem-1}


\begin{proof}
[Proof of Lemma~\ref{lemma:TV}]
By \eqref{eq:fluid-BV}, $\bar{B}_{n}(t+\bar{V}_{n}(t))\leq\bar{B}_{n}(\bar
{V}_{n}(0))+\bar{E}_{n}(T)$ for $0\leq t\leq T$. Because $\bar{B}_{n}$ have
nondecreasing sample paths,
\[
\mathbb{P}\Big[\sup_{0\leq t\leq T}(t+\bar{V}_{n}(t))>a\Big]\leq
\mathbb{P}[\bar{B}_{n}(a)\leq\bar{B}_{n}(\bar{V}_{n}(0))+\bar{E}_{n}%
(T)]\quad\mbox{for }a>0.
\]
By \eqref{eq:initial}, \eqref{eq:VW}, and Proposition~\ref{prop:FSLLN},
$\bar{B}_{n}(\bar{V}_{n}(0)) \Rightarrow\mu \bar{w}$ as $ n\rightarrow\infty $.
Then, \eqref{eq:V-bar-bounded} follows from \eqref{eq:fluid-E}.

Because $e+\bar{V}_{n}$ has nondecreasing sample paths (see Lemma~3.3 in
\citet{DaiHe10}), we have $
(\bar{V}_{n}(t_{2})-\bar{V}_{n}(t_{1}))^{+}\leq(t_{2}+\bar{V}_{n}
(t_{2}))-(t_{1}+\bar{V}_{n}(t_{1}))$ and $(\bar{V}_{n}
(t_{2})-\bar{V}_{n}(t_{1}))^{-}\leq t_{2}-t_{1}$ for $0\leq t_{1}\leq t_{2}$. Hence, the total variation of $\bar
{V}_{n}$ over $[0,T]$ satisfies%
\[
\int_{0}^{T}|\mathrm{d}\bar{V}_{n}(t)|\leq 2T+\bar{V}_{n}(T)-\bar{V}_{n}(0).
\]
Then, \eqref{eq:V-total} follows from \eqref{eq:V-bar-bounded}.
\end{proof}

Martingale arguments are extensively involved in subsequent proofs. Let us
define the associated filtrations. In the $n$th system, let $v_{n,k}$ be the
service time of the $k$th customer arriving after time zero. Note that $v_{n,k}$
is not identical to $\xi_{n,k}$, since the latter is the $ k $th service time
finished by the $ n $th server. To keep track of the history of the queue, we
define a filtration $\{\mathcal{F}_{n,i}:i=0,1,\ldots\}$ by
\[
\mathcal{F}_{n,i}=\sigma\{a_{n,k+1},v_{n,k},\zeta_{n,k}:k\leq i\},
\]
where $ a_{n,k} $ and $ \zeta_{n,k} $ are the arrival and patience times of the
$ k $th customer. By Lemma~3.1 in \citet{DaiHe10}, $V_{n}(a_{n,k}-)$ is
$\mathcal{F}_{n,i} $-measurable for all $k\leq i+1$. Modifying the above
filtration, we construct a continuous-time filtration
$\{\mathcal{F}_{n}(t):t\geq0\}$ by
\[
\mathcal{F}_{n}(t)=\mathcal{F}_{n,\lfloor n\gamma_{n}t\rfloor}.
\]

\begin{proof}
	[Proof of Lemma~\ref{lemma:fluid-R}]
Put
\[
\bar{Z}_{n}(t)=\frac{1}{n\gamma_{n}}\sum_{k=1}^{\lfloor n\gamma_{n}t\rfloor
}\big(1_{\{\zeta_{n,k}\leq V_{n}(a_{n,k}-)\}}-\Theta_{n}(V_{n}(a_{n,k}-))\big).
\]
Then, $\{(\bar{Z}_{n}(t),\mathcal{F}_{n}(t)):t\geq0\}$ is a martingale with quadratic variation
\[
[\bar{Z}_{n}](t)=\frac{1}{n^{2}\gamma_{n}^{2}}\sum_{k=1}^{\lfloor
n\gamma_{n}t\rfloor}\big(1_{\{\zeta_{n,k}\leq V_{n}(a_{n,k}-)\}}-\Theta_{n}%
(V_{n}(a_{n,k}-))\big)^{2}
\]
(see Lemma~4.2 in \citet{DaiHe10}). Clearly, $\lim_{n\rightarrow\infty}[\bar{Z}_{n}](t)=0$ for $t>0$.
We also have
\[
\lim_{n\rightarrow\infty}\sup_{0\leq t\leq T}|\Delta\bar{Z}_{n}(t)|=
0
\]
because $|\Delta\bar{Z}_{n}(t)|\leq1/(n\gamma_{n})$. Then,
it follows from the martingale FCLT (see Theorem~7.1.4 in \citet{EthierKurtz86}) that
$\bar{Z}_{n}\Rightarrow0$ as $n\rightarrow\infty$. By \eqref{eq:fluid-E} and
the random time-change theorem,
\begin{equation}
\bar{Z}_{n}\circ\bar{E}_{n}\Rightarrow0\quad\mbox{as }n\rightarrow\infty.
\label{eq:ZE}
\end{equation}

Put
\[
\Gamma_{n}(t)=\sum_{k=1}^{E_{n}(t)}\Theta_{n}(V_{n}(a_{n,k}-))\quad\mbox{and}\quad\bar{\Gamma}_{n}(t)=\frac{1}{n\gamma_{n}}\Gamma_{n}(\gamma_{n}t).
\]
By \eqref{eq:Theta-n} and \eqref{eq:bar-VR}, $\bar{\Gamma}_{n}(t)$ can be written into a
Riemann--Stieltjes integral
\[
\bar{\Gamma}_{n}(t)=\int_{0}^{t}H(\bar{V}_{n}(u-))\,\mathrm{d}\bar{E}_{n}(u).
\]
Using integration by parts, we obtain
\[
\bar{\Gamma}_{n}(t)-\rho\mu\int_{0}^{t}H(\bar{V}_{n}(u))\,\mathrm{d}u=(\bar{E}%
_{n}(t)-\rho\mu t)H(\bar{V}_{n}(t))-\int_{0}^{t}(\bar{E}_{n}(u)-\rho\mu
u)\,\mathrm{d}H(\bar{V}_{n}(u)).
\]
By \eqref{eq:fluid-E}, the first term on the right side satisfies%
\[
\sup_{0\leq t\leq T}\left\vert (\bar{E}_{n}(t)-\rho\mu t)H(\bar{V}%
_{n}(t))\right\vert \Rightarrow0\quad\mbox{as }n\rightarrow\infty.
\]
For the second term, it follows from \eqref{eq:bounded-density} that
\begin{align*}
\sup_{0\leq t\leq T}\Big\vert \int_{0}^{t}(\bar{E}_{n}(u)-\rho\mu
u)\,\mathrm{d}H(\bar{V}_{n}(u))\Big\vert  &  \leq\sup_{0\leq t\leq
T}\vert \bar{E}_{n}(t)-\rho\mu t\vert \int_{0}^{T}|\mathrm{d}%
H(\bar{V}_{n}(t))|\\
&  \leq\kappa\sup_{0\leq t\leq T}\vert \bar{E}_{n}(t)-\rho\mu
t\vert \int_{0}^{T}|\mathrm{d}\bar{V}_{n}(t)|.
\end{align*}
By \eqref{eq:fluid-E} and \eqref{eq:V-total},
\[
\sup_{0\leq t\leq T}\Big\vert \int_{0}^{t}(\bar{E}_{n}(u)-\rho\mu
u)\,\mathrm{d}H(\bar{V}_{n}(u))\Big\vert \Rightarrow0\quad
\mbox{as }n\rightarrow\infty.
\]
It follows that
\begin{equation}
\sup_{0\leq t\leq T}\Big\vert\bar{\Gamma}_{n}(t)-\rho\mu\int_{0}^{t}H(\bar{V}%
_{n}(u))\,\mathrm{d}u\Big\vert\Rightarrow0\quad\mbox{as }n\rightarrow
\infty.\label{eq:bar-S}%
\end{equation}
Because $\bar{R}_{n}(t)=\bar{Z}_{n}(\bar{E}_{n}(t))+\bar{\Gamma}_{n}(t)$, the
lemma follows from \eqref{eq:ZE} and \eqref{eq:bar-S}.
\end{proof}

\begin{proof}
	[Proof of Lemma~\ref{lemma:EBR}]
Put
\[
\tilde{Z}_{n}(t)=\frac{1}{\sqrt{n\gamma_{n}}}\sum_{k=1}^{\lfloor n\gamma
	_{n}t\rfloor}\big(1_{\{\bar{\zeta}_{n,k}\leq\bar{V}_{n}(\bar{a}_{n,k}-)\}}
-H(\bar{V}_{n}(\bar{a}_{n,k}-))\big)
\]
and
\[
\tilde{Z}^{\prime}_{n}(t)=\frac{1}{\sqrt{n\gamma_{n}}}\sum_{k=1}^{\lfloor n\gamma_{n}t\rfloor}\big(1_{\{\bar{\zeta}_{n,k}\leq \bar{w}\}}-H(\bar{w})\big).
\]
Clearly, $\mathbb{P}[\bar{\zeta}_{n,k}\leq
\bar{w}]=H(\bar{w})=(\rho-1)/\rho$ for $n,k\in\mathbb{N}$. Then,
$\tilde{Z}^{\prime}_{n}\Rightarrow\hat{Z}$ as $n\rightarrow\infty$ by Donsker's
theorem, where $\hat{Z}$ is a driftless Brownian motion with variance
$(\rho-1)/\rho^{2}$. Because $\tilde{E}_{n}$, $\tilde{B}_{n}$, and
$\tilde{Z}^{\prime}_{n}$ are mutually independent, it follows from \eqref{eq:E}
and Theorem~\ref{theorem:FCLT} that
\begin{equation}
(\tilde{E}_{n},\tilde{B}_{n},\tilde{Z}^{\prime}_{n})\Rightarrow(\hat{E},\hat
{B},\hat{Z})\quad\mbox{as }n\rightarrow\infty,
\label{eq:EBZ}
\end{equation}
where $\hat{E}$, $\hat{B}$, and $\hat{Z}$ are also mutually independent.

By Lemma~4.2 in \citet{DaiHe10}, $\{(\tilde{Z}_{n}(t)-\tilde{Z}^{\prime}_{n}
(t),\mathcal{F}_{n}(t)):t\geq0\}$ is a martingale with
\[
[\tilde{Z}_{n}-\tilde{Z}^{\prime}_{n}](t)=\frac{1}{n\gamma_{n}}\sum_{k=1}^{\lfloor n\gamma_{n}t\rfloor}\big(1_{\{\bar{\zeta}_{n,k}\leq\bar{V}_{n}(\bar{a}_{n,k} -)\}}-1_{\{\bar{\zeta}_{n,i}\leq \bar{w}\}} + H(\bar{w})-H(\bar{V}_{n}(\bar{a}_{n,k} -))\big)^{2}.
\]
Because $\bar{\zeta}_{n,k}$ is independent of $\mathcal{F}_{n,k-1}$ and $\bar
{V}_{n}(\bar{a}_{n,k}-)$ is $\mathcal{F}_{n,k-1}$-measurable,
\begin{align*}
&  \mathbb{E}\Big[\big(1_{\{\bar{\zeta}_{n,k}\leq\bar{V}_{n}(\bar{a}_{n,k}
-)\}}-1_{\{\bar{\zeta}_{n,k}\leq  \bar{w}\}}+H(\bar{w})-H(\bar{V}_{n}(\bar{a}_{n,k}
-))\big)^{2} \big\vert \mathcal{F}_{n,k-1}\Big]\\
&\quad  =H(\bar{V}_{n}(\bar{a}_{n,k}-))+H(\bar{w})-2H(\bar{V}_{n}(\bar{a}%
_{n,k}-)\wedge  \bar{w})-\big(H(\bar{w})-H(\bar{V}_{n}(\bar{a}_{n,k}-))\big)^{2}\\
&\quad  \leq H(\bar{V}_{n}(\bar{a}_{n,k}-))+H(\bar{w})-2H(\bar{V}_{n}(\bar{a}%
_{n,k}-)\wedge \bar{w})\\
&\quad  =|H(\bar{V}_{n}(\bar{a}_{n,k}-))-H(\bar{w})|.
\end{align*}
Hence,
\[
\mathbb{E\big[}[\tilde{Z}_{n}-\tilde{Z}^{\prime}_{n}](t)\big]
\leq \frac{1}{n\gamma_{n}
}\sum_{k=1}^{\lfloor n\gamma_{n}t\rfloor}\mathbb{E}[|H(\bar{V}_{n}(\bar
{a}_{n,k}-))-H(\bar{w})|]\leq t\cdot\mathbb{E}\Big[\sup_{1\leq k\leq n\gamma_{n}%
t}|H(\bar{V}_{n}(\bar{a}_{n,k}-))-H(\bar{w})|\Big].
\]
For any $\varepsilon>0$ and $ t>0 $, it follows from \eqref{eq:bounded-density} that
\begin{align*}
&  \mathbb{P}\Big[\sup_{1\leq k\leq n\gamma_{n}t}|H(\bar{V}_{n}(\bar{a}%
_{n,k}-))-H(\bar{w})|>\varepsilon\Big]\\
& \quad  \leq\mathbb{P}\Big[\sup_{1\leq k\leq n\gamma_{n}t}|\bar{V}_{n}(\bar
{a}_{n,k}-)-\bar{w}|>\frac{\varepsilon}{\kappa}\Big]\\
& \quad  \leq\mathbb{P}\Big[\sup_{1\leq k\leq E_{n}(\mu^{-1}\gamma_{n}t)}|\bar
{V}_{n}(\bar{a}_{n,k}-)-\bar{w}|>\frac{\varepsilon}{\kappa}\Big]+\mathbb{P}[E_{n}(\mu^{-1}\gamma
_{n}t)<n\gamma_{n}t]\\
& \quad  \leq\mathbb{P}\Big[\sup_{1\leq u\leq\mu^{-1}t}|\bar{V}_{n}(u)-\bar{w}|>\frac{\varepsilon}{\kappa}
\Big]+\mathbb{P}[\bar{E}_{n}(\mu^{-1}t)<t].
\end{align*}
Then, we deduce from \eqref{eq:fluid-E} and Proposition~\ref{prop:fluid-V} that
\[
\lim_{n\rightarrow\infty}\mathbb{P}\Big[\sup_{1\leq k\leq n\gamma_{n}t}%
|H(\bar{V}_{n}(\bar{a}_{n,k}-))-H(\bar{w})|>\varepsilon\Big]=0.
\]
Because $|H(\bar{V}_{n}(\bar{a}_{n,k}-))-H(\bar{w})|\leq 1$, we further obtain
\[
\lim_{n\rightarrow\infty}\mathbb{E}\Big[\sup_{1\leq k\leq n\gamma_{n}t}%
|H(\bar{V}_{n}(\bar{a}_{n,k}-))-H(\bar{w})|\Big]=0.
\]
Then,
$\mathbb{E\big[}[\tilde{Z}_{n}-\tilde{Z}^{\prime}_{n}](t)\big]\rightarrow0$ and thus
$[\tilde{Z}_{n}-\tilde{Z}^{\prime}_{n}](t)\Rightarrow0$ as $n\rightarrow\infty$ for all
$t>0$. In addition, because $|\Delta(\tilde{Z}_{n}-\tilde{Z}^{\prime}_{n})(t)|
\leq 2/\sqrt{n\gamma_{n}}$, it follows from the martingale FCLT that $\tilde{Z}_{n}-\tilde{Z}^{\prime}_{n}\Rightarrow0$ as $n\rightarrow\infty$, which, along with \eqref{eq:EBZ}
and the convergence-together theorem, implies that
\[
(\tilde{E}_{n},\tilde{B}_{n},\tilde{Z}_{n})\Rightarrow(\hat{E},\hat
{B},\hat{Z})\quad\mbox{as }n\rightarrow\infty.
\]
Because $\tilde{R}_{n}=\tilde{Z}_{n}\circ\bar{E}_{n}$, we conclude the
proof using \eqref{eq:fluid-E} and the random time-change theorem.
\end{proof}

\begin{proof}
	[Proof of Lemma~\ref{lemma:HV}]
Put
\[
\tilde{\Upsilon}_{n}^{\prime}(t)=\frac{1}{\mu\sqrt{n\gamma_{n}}}\sum
_{k=1}^{E_{n}(\gamma_{n}t)}\big(H(\bar{V}_{n}(\bar{a}_{n,k}-))-H(\bar{w})\big)- \sqrt{n\gamma_{n}} \rho \int_{0}^{t}\big(H(\bar{V}_{n}(u))-H(\bar{w})\big)\,\mathrm{d}u.
\]
By \eqref{eq:bar-VR} and integration by parts,
\begin{align*}
\tilde{\Upsilon}_{n}^{\prime}(t)  &  =\frac{\sqrt{n\gamma_{n}}}{\mu}\int%
_{0}^{t}\big(H(\bar{V}_{n}(u-))-H(\bar{w})\big)\,\mathrm{d}\bar{E}_{n}%
(u)-\sqrt{n\gamma_{n}}\rho\int_{0}^{t}\big(H(\bar{V}_{n}%
(u))-H(\bar{w})\big)\,\mathrm{d}u\\ &  =\frac{\sqrt{n\gamma_{n}}}{\mu}
\big(H(\bar{V}_{n}(t))-H(\bar{w})\big)\big(\bar {E}_{n}(t)-\rho\mu
t\big)-\frac{\sqrt{n\gamma_{n}}}{\mu}\int_{0}^{t}%
\big(\bar{E}_{n}(t)-\rho\mu t\big)\,\mathrm{d}H(\bar{V}_{n}(u))\\ &
=\frac{1}{\mu}\big(H(\bar{V}_{n}(t)) -H(\bar{w})\big)\tilde{E}_{n}(t)-\frac{1}%
{\mu}\int_{0}^{t}\tilde{E}_{n}(u-)\,\mathrm{d}H(\bar{V}_{n}(u))
+\frac{1}{\mu}\sum_{0<u\leq
	t}\Delta\tilde{E}_{n}(u)\Delta H(\bar{V}_{n}(u)).
\end{align*}
Consider the three terms on the right side. By \eqref{eq:E} and
Proposition~\ref{prop:fluid-V},%
\[
\sup_{0\leq t\leq T}\big\vert\big(H(\bar{V}_{n}(t))-H(\bar{w})\big)\tilde{E}%
_{n}(t)\big\vert\Rightarrow0\quad\text{as }n\rightarrow\infty.
\]
Using the continuous mapping theorem and Theorem~3.9 in \citet{Billingsley99},
we obtain
\[
(\tilde{E}_{n},H\circ\bar{V}_{n})\Rightarrow(\hat{E},H(\bar{w})\chi)\quad
\mbox{as }n\rightarrow\infty.
\]
By \eqref{eq:bounded-density},
$ \int_{0}^{t}|\mathrm{d}H(\bar{V}_{n}(u))|\leq\kappa\int_{0}^{t}|\mathrm{d}%
\bar{V}_{n}(u)|$.
Then, it follows from \eqref{eq:V-total} that
\begin{equation}
\lim_{a\rightarrow\infty}\limsup_{n\rightarrow\infty}\mathbb{P}\Big[\int%
_{0}^{T}|\mathrm{d}H(\bar{V}_{n}(t))|>a\Big]=0, \label{eq:dHV}%
\end{equation}
which implies that $\{H\circ\bar{V}_{n}:n\in\mathbb{N}\}$ is uniformly tight
(see Definition~7.4 in \citet{KurtzProtter96} for the definition of uniform
tightness). Because $\int_{0}^{t}H(\bar{w})\hat{E}(u-)\,\mathrm{d}\chi(u)=0$
for all $t>0$,
it follows from Theorem~7.10 in \citet{KurtzProtter96} that
\[
\sup_{0\leq t\leq T}\Big\vert\int_{0}^{t}\tilde{E}_{n}(u-)\,\mathrm{d}%
H(\bar{V}_{n}(u))\Big\vert\Rightarrow0\quad\text{as }n\rightarrow\infty.
\]
In addition, the third term satisfies
\[
\sup_{0\leq t\leq T}\Big\vert\sum_{0<t\leq T}\Delta\tilde{E}_{n}(t)\Delta
H(\bar{V}_{n}(t))\Big\vert\leq\sup_{0\leq t\leq T}|\Delta\tilde{E}_{n}%
(t)|\int_{0}^{T}|\mathrm{d}H(\bar{V}_{n}(t))|.
\]
By \eqref{eq:E}, \eqref{eq:dHV}, and the fact that $\hat{E}$ has almost sure
continuous sample paths,
\[
\sup_{0\leq t\leq T}\Big\vert\sum_{0<t\leq T}\Delta\tilde{E}_{n}(t)\Delta
H(\bar{V}_{n}(t))\Big\vert\Rightarrow0\quad\text{as }n\rightarrow\infty.
\]
Therefore, $\tilde{\Upsilon}_{n}^{\prime}\Rightarrow0$ as $n\rightarrow\infty$.

Rewrite \eqref{eq:Vn-diffusion} into%
\[
\tilde{V}_{n}(t)=\tilde{V}_{n}(0)+\tilde{M}_{n}(t)-\tilde{\Upsilon}%
_{n}^{\prime}(t)-\sqrt{n\gamma_{n}}\rho\int_{0}^{t}\big(H(\bar{V}%
_{n}(u))-H(\bar{w})\big)\,\mathrm{d}u.
\]
It follows from \eqref{eq:bounded-density} and \eqref{eq:V-tilde} that
\begin{align*}
|\tilde{V}_{n}(t)|  &  \leq|\tilde{V}_{n}(0)+\tilde{M}_{n}(t)-\tilde{\Upsilon
}_{n}^{\prime}(t)|+\sqrt{n\gamma_{n}}\rho\int_{0}^{t}|H(\bar{V}_{n}%
(u))-H(\bar{w})|\,\mathrm{d}u\\
&  \leq|\tilde{V}_{n}(0)+\tilde{M}_{n}(t)-\tilde{\Upsilon}_{n}^{\prime
}(t)|+\kappa\rho\int_{0}^{t}|\tilde{V}_{n}(u)|\,\mathrm{d}u.
\end{align*}
Using Gronwall's inequality (see Lemma~21.4 in \citet{Kallenberg02}), we
obtain
\[
|\tilde{V}_{n}(t)|\leq|\tilde{V}_{n}(0)+\tilde{M}_{n}(t)-\tilde{\Upsilon}%
_{n}^{\prime}(t)|\exp(\kappa\rho t).
\]
By \eqref{eq:initial}, \eqref{eq:tilde-WV}, \eqref{eq:M-tilde}, Lemma~\ref{lemma:EBR}, and the convergence of
$\tilde{\Upsilon}_{n}^{\prime}$, we have%
\[
\lim_{a\rightarrow\infty}\limsup_{n\rightarrow\infty}\mathbb{P}\Big[\sup
_{0\leq t\leq T}|\tilde{V}_{n}(0)+\tilde{M}_{n}(t)-\tilde{\Upsilon}^{\prime}
_{n}(t)|>a\Big]=0,
\]
which implies that $\{\tilde{V}_{n}:n\in\mathbb{N}\}$ is stochastically bounded.

Put
\[
\varphi(t)=%
\begin{cases}
(t-\bar{w})^{-1}(H(t)-H(\bar{w})) & \text{for }t\neq \bar{w},\\
f_{H}(\bar{w}) & \text{for }t=\bar{w}.
\end{cases}
\]
By \eqref{eq:bar-VR} and \eqref{eq:V-tilde},
\[
\sqrt{n\gamma_{n}}\int_{0}^{t}\big(H(\bar{V}_{n}%
(u))-H(\bar{w})\big)\,\mathrm{d}u- f_{H}(\bar{w})\int_{0}^{t}\tilde{V}_{n}(u)\,\mathrm{d}u
=\int_{0}^{t}\big(\varphi(\bar{V}_{n}(u))-f_{H}(\bar{w})\big)\tilde{V}%
_{n}(u)\,\mathrm{d}u.
\]
Because $\lim_{u\rightarrow\bar{w}}\varphi(u)=f_{H}(\bar{w})$ and $\{\tilde{V}_{n}:n\in\mathbb{N}\}$ is stochastically bounded, it follows from
Proposition~\ref{prop:fluid-V} that
\[
\sup_{0\leq t\leq T}\Big\vert\sqrt{n\gamma_{n}}\int_{0}^{t}\big(H(\bar
{V}_{n}(u))-H(\bar{w})\big)\,\mathrm{d}u-f_{H}(\bar{w})\int_{0}^{t}\tilde{V}_{n}
(u)\,\mathrm{d}u\Big\vert\Rightarrow0\quad\text{as }n\rightarrow\infty,
\]
which, along with the convergence of $\tilde{\Upsilon}_{n}^{\prime}$, leads to
the assertion of the lemma.
\end{proof}

\subsection{Proofs of lemmas for Theorem~\ref{theorem:many-server}}
\label{sec:lemmas-Theorem-2}

\begin{proof}
	[Proof of Lemma~\ref{lemma:WGG}]
Using integration by parts, we obtain
\[
\tilde{G}_{n}^{\prime}(t)=\int_{t}^{t+\bar{w}}\tilde{E}_{n}(u)\,\mathrm{d}
H(t+\bar{w}-u)+\tilde{E}_{n}(t+\bar{w})-(1-H(\bar{w}))\tilde{E}_{n}(t).
\]
It follows from \eqref{eq:E} that

\[
\tilde{G}_{n}^{\prime}(t)\Rightarrow\int_{t}^{t+\bar{w}}\hat{E}(u)\,\mathrm{d}
H(t+\bar{w}-u)+\hat{E}(t+\bar{w})-(1-H(\bar{w}))\hat{E}(t)\quad\text{as
}n\rightarrow\infty.
\]
Using integration by parts again, we have
\[
\int_{t}^{t+\bar{w}}\hat{E}(u)\,\mathrm{d}H(t+\bar{w}-u)+\hat{E}(t+\bar
{w})-(1-H(\bar{w}))\hat{E}(t)=\int_{t}^{t+\bar{w}}(1-H(t+\bar{w}
-u))\,\mathrm{d}\hat{E}(u).
\]
Then, $\tilde{G}_{n}^{\prime}(t)\Rightarrow\hat{G}^{\prime}(t)$ as
$n\rightarrow\infty$.

Write $a_{n,k}^{\prime}=a_{n,k+E_{n}(\gamma_{n}t)}$ and $\zeta_{n,k}^{\prime
}=\zeta_{n,k+E_{n}(\gamma_{n}t)}$, which are the arrival and patience times of
the $k$th customer arriving after $\gamma_{n}t$, respectively. Put
\[
\tilde{K}_{n}(u)=\frac{1}{\sqrt{n\gamma_{n}}}\sum_{k=1}^{\lfloor n\gamma
	_{n}u\rfloor}\big(H(\bar{w}-\bar{b}_{n,k})-1_{\{\bar{\zeta}_{n,k}^{\prime}\leq
	\bar{w}-\bar{b}_{n,k}\}}\big)
\]
with $\bar{b}_{n,k}=k/(\rho\mu n\gamma_{n})$. Then, $\tilde{K}_{n}$ is a
martingale with quadratic variation%
\[
\lbrack\tilde{K}_{n}](u)=\frac{1}{n\gamma_{n}}\sum_{k=1}^{\lfloor n\gamma
	_{n}u\rfloor}\big(  H(\bar{w}-\bar{b}_{n,k})-1_{\{\bar{\zeta}_{n,k}^{\prime
	}\leq\bar{w}-\bar{b}_{n,k}\}}\big)^{2}.
\]
Because
\[
\mathbb{E}\big[\big(H(\bar{w}-\bar{b}_{n,k})-1_{\{\bar{\zeta}_{n,k}^{\prime}\leq
	\bar{w}-\bar{b}_{n,k}\}}\big)^{2}\big]=H(\bar{w}-\bar{b}_{n,k})(1-H(\bar{w}-\bar
{b}_{n,k})),
\]
by the weak law of large numbers (see, e.g., Theorem~5.14 of \citet{Klenke14}),
we obtain%
\[
\lbrack\tilde{K}_{n}](u)\Rightarrow\rho\mu\int_{0}^{u/\rho\mu}H(\bar
{w}-s)(1-H(\bar{w}-s))\,\mathrm{d}s.
\]
Since $|\Delta\tilde{K}_{n}(u)|\leq1/\sqrt{n\gamma_{n}}$ for $u\geq0$, we
further have%
\[
\lim_{n\rightarrow\infty}\sup_{0\leq u\leq T}|\Delta\tilde{K}_{n}%
(u)|=0\quad\text{for }T>0.
\]
Then, it follows from the martingale FCLT that
$ \tilde{K}_{n}\Rightarrow\hat{K}$ as $n\rightarrow\infty$,
where
\[
\hat{K}(u)=\int_{0}^{u/\rho\mu}\sqrt{\rho\mu H(\bar{w}-s)(1-H(\bar{w}%
	-s))}\,\mathrm{d}\hat{S}(s).
\]
Put%
\[
\tilde{K}_{n}^{\prime}(u) = \frac{1}{\sqrt{n\gamma_{n}}} \sum_{k=1}^{\lfloor
	n\gamma_{n}u\rfloor}\big(H(\bar{W}_{n}(t)-\bar{a}_{n,k}^{\prime}) -
1_{\{\bar{\zeta}_{n,k}^{\prime} \leq \bar{W}_{n}(t)-\bar{a}_{n,k}^{\prime}\}}\big).
\]
By similar arguments as in the proof of Lemma~\ref{lemma:EBR}, we can prove that
$\tilde{K}_{n}-\tilde{K}_{n}^{\prime}\Rightarrow0$ as $n\rightarrow\infty$.

By \eqref{eq:Vn-diffusion}, Lemmas~\ref{lemma:EBR} and~\ref{lemma:HV}, Proposition~\ref{prop:virtual-limit-perturbed}, and the asymptotic equivalence between $\tilde {V}_{n}$ and $\tilde{W}_{n}$, we deduce that
$(\tilde{E}_{n},\tilde{W}_{n})\Rightarrow(\hat{E},\hat{W})$, so that $(\tilde{W}_{n},\tilde{G}_{n}^{\prime})\Rightarrow
(\hat{W},\hat{G}^{\prime})$ as $n\rightarrow\infty$. Because $\tilde
{K}_{n}$ is independent of $(\tilde{W}_{n}(t),\tilde{G}_{n}^{\prime}(t))$,
we further obtain
\[
(\tilde{W}_{n}(t)\chi,\tilde{G}_{n}^{\prime}(t)\chi,\tilde{K}_{n})\Rightarrow
(\hat{W}(t)\chi,\hat{G}^{\prime}(t)\chi,\hat{K})\quad\text{as }n\rightarrow\infty,
\]
where $\hat{W}(t)$, $\hat{G}(t)$, and $\hat{K}$ are mutually independent.
Because $\tilde{K}_{n}-\tilde{K}_{n}^{\prime}\Rightarrow0$ as $n\rightarrow
\infty$ and
\[
\tilde{G}_{n}''(t)=\tilde{K}_{n}^{\prime}(\bar{E}_{n}(t+\bar{W}%
_{n}(t))-\bar{E}_{n}(t)),
\]
the assertion of the lemma follows from \eqref{eq:fluid-E}, \eqref{eq:Wn-bar-limit}, and the random time-change theorem.
\end{proof}

\begin{proof}
	[Proof of Lemma~\ref{lemma:YY}]
Since $\tilde{E}_{n}\Rightarrow\hat{E}$ and $\hat{E}$
has almost sure continuous sample paths, then by \eqref{eq:Wn-bar-limit},
\[
\tilde{E}_{n}(t+\bar{W}_{n}(t))-\tilde{E}_{n}(t+\bar{w})\Rightarrow
0\quad\text{as }n\rightarrow\infty.
\]
Because
\[
\Big\vert\sqrt{n\gamma_{n}}\int_{t+\bar{w}}^{t+\bar{W}_{n}(t)}
H(t+\bar{W}_{n}(t)-u)\,\mathrm{d}\bar{E}_{n}(u)\Big\vert
\leq H(\vert \bar{W}_{n}(t)-\bar{w}\vert )\big\vert \tilde{W}_{n}(t)\big(\bar{E}_{n}(t+\bar{W}_{n}(t))-
\bar{E}_{n}(t+\bar{w})\big)\big\vert ,
\]
it follows from \eqref{eq:fluid-E}, \eqref{eq:Wn-bar-limit}, and Theorem~\ref{theorem:many-server-virtual} that
\[
\sqrt{n\gamma_{n}}\int_{t+\bar{w}}^{t+\bar{W}_{n}(t)}H(t+\bar{W}_{n}(t)-u)
\,\mathrm{d} \bar{E}_{n}(u) \Rightarrow 0 \quad \mbox{as }n\rightarrow
\infty.
\]
Write
\[
\varphi(u,\delta)=%
\begin{cases}
\delta^{-1}(H(u+\delta)-H(u)) & \text{for }\delta\neq0,\\
f_{H}(u) & \text{for }\delta=0.
\end{cases}
\]
Then, $\lim_{\delta\rightarrow0}\varphi(u,\delta)=f_{H}(u)$ for
$u\geq0$. Using the fact that $H(\bar{w})=(\rho-1)/\rho$, we obtain%
\begin{align*}
&  (\rho-1)\mu\tilde{W}_{n}(t)-\sqrt{n\gamma_{n}}\rho\mu\int_{0}^{\bar{w}
}(H(\bar{W}_{n}(t)-u)-H(\bar{w}-u))\,\mathrm{d}u\\
&  =\rho\mu\tilde{W}_{n}(t)\int_{0}^{\bar{w}}(  f_{H}(\bar{w}
-u)-\varphi(\bar{w}-u,\bar{W}_{n}(t)-\bar{w}))  \,\mathrm{d}u.
\end{align*}
By \eqref{eq:bounded-density}, we have $|f_{H}(\bar{w}-u) -
\varphi(\bar{w}-u, \bar{W}_{n}(t) - \bar {w})| \leq 2\kappa$. It follows
from \eqref{eq:Wn-bar-limit}, Theorem~\ref{theorem:many-server-virtual}, and the dominated convergence theorem that
\[
(\rho-1)\mu\tilde{W}_{n}(t)-\sqrt{n\gamma_{n}}\rho\mu\int_{0}^{\bar{w}}%
(H(\bar{W}_{n}(t)-u)-H(\bar{w}-u))\,\mathrm{d}u\Rightarrow0\quad\text{as
}n\rightarrow\infty.
\]
Hence, $\tilde{Y}_{n}(t)\Rightarrow0$ as $n\rightarrow\infty$.

Put $\check{H}_{n}(u)=H(t+\bar{w}-u)-H(t+\bar{W}_{n}(t)-u)$ for $u\geq0$. By
\eqref{eq:Wn-bar-limit} and Theorem~3.9 in \citet{Billingsley99}, we obtain
\[
(\tilde{E}_{n},\check{H}_{n})\Rightarrow(\hat{E},0)\quad\mbox{as }
n\rightarrow\infty.
\]
By similar arguments as in the proof of Lemma~\ref{lemma:HV}, we can show that
$\{\check{H}_{n}:n\in\mathbb{N}\}$ is uniformly tight. Using integration by
parts, we obtain%
\[
\tilde{Y}_{n}^{\prime}(t)=\int_{t}^{t+\bar{w}}\tilde{E}_{n}(u)\,\mathrm{d}%
\check{H}_{n}(u)+\check{H}_{n}(t+\bar{w})\tilde{E}_{n}(t+\bar{w})-\check
{H}_{n}(t)\tilde{E}_{n}(t).
\]
Then, it follows from Theorem~7.10 in \citet{KurtzProtter96} that
\[
\int_{t}^{t+\bar{w}}\tilde{E}_{n}(u)\,\mathrm{d}\check{H}_{n}(u)\Rightarrow
0\quad\text{as }n\rightarrow\infty,
\]
which, along with \eqref{eq:E}, implies that
$\tilde{Y}_{n}^{\prime}(t)\Rightarrow 0$ as $n\rightarrow\infty$. The joint convergence follows from Theorem~3.9 in \citet{Billingsley99}.
\end{proof}

\section*{Acknowledgements}
This research was supported in part by MOE AcRF Grant R-266-000-063-133 and NUS
Global Asia Institute Grant R-716-000-006-133.

\bibliographystyle{ormsv080}
\bibliography{../refs}
\end{document}